\documentclass[11pt,leqno,a4paper, english]{amsart}

\usepackage{tikz}
\usepackage[T1]{fontenc}
\usepackage{amsmath}
\usepackage{amsfonts,amssymb}
\usepackage{enumerate}
\usepackage{mathrsfs}
\usepackage{amsthm}
\renewcommand{\Re} { \mathbf{Re}}
\renewcommand{\Im} {\mathbf{Im}}

\usepackage{hyperref}
\hypersetup{
    colorlinks=true,       
    linkcolor=red,          
    citecolor=blue,        
    filecolor=magenta,      
    urlcolor=cyan           
}
    \usepackage[all]{hypcap}
\newtheoremstyle{note}{} {}{\itshape}{-6pt}{\bf}{. --}{ }{}

\theoremstyle{note}
\newtheorem{theo}{Theorem}
\newtheorem{prop}{Proposition}[section]
\newtheorem{lemm}[prop]{Lemma}
\newtheorem{coro}[prop]{Corollary}
\newtheorem{assu}[prop]{Assumption}

\newtheorem{rema}[prop]{Remark}

\newcommand{\drawcarre}[3]{
\draw[ black, fill=#3] (#1,#2) -- (#1+1,#2) -- (#1+1,#2+1)--(#1,#2+1)--cycle;}

\numberwithin{equation}{section}


\title[Stabilisation with rough dampings]{Stabilisation of wave equations on the torus with rough dampings}
\author{N. Burq}
\address{Laboratoire de Math\'ematiques d'Orsay, CNRS, Universit\'e Paris-Saclay, B\^atiment~307, 91405 Orsay Cedex et Institut Universitaire de France}

 \email{nicolas.burq@math.u-psud.fr}
 \author{P. G\'erard }
 \address{Laboratoire de Math\'ematiques d'Orsay, CNRS, Universit\'e Paris-Saclay, B\^atiment~307, 91405 Orsay Cedex} 
 \email{patrick.gerard@math.u-psud.fr}
 \date{December 6, 2019}
\begin{document}

\begin{abstract}  For the damped wave equation on a compact manifold with {\em continuous} dampings, the geometric control condition is necessary and sufficient for {uniform} stabilisation. In this article, on the two dimensional torus, in the special case where  $a(x) = \sum_{j=1}^N a_j 1_{x\in R_j}$ ($R_j$ are polygons), we give a very simple necessary and sufficient geometric condition for uniform stabilisation. We also propose a natural generalization of the geometric control condition which makes sense for $L^\infty$ dampings. We show that this condition is always necessary for uniform stabilisation (for any compact (smooth) manifold and any $L^\infty$ damping), and we prove that it is sufficient in our particular case on $\mathbb{T}^2$ (and for our particular dampings).
\\
\ \vskip .1cm 
\noindent {\sc R\'esum\'e.} Pour l'\'equation des ondes amortie sur une vari\'et\'e compacte, dans le cas d'un amortissement {\em continu}, la condition de contr\^ole g\'eom\'etrique est n\'ecessaire et suffisante pour la stabilisation uniforme. Dans cet article,  sur le tore $\mathbb{T}^2$ et dans le cas o\`u  $a(x) = \sum_{j=1}^N a_j 1_{x\in R_j}$ ($R_j$ sont des polygones), nous exhibons une condition g\'eom\'etrique n\'ecessaire et suffisante tr\`es simple. Nous proposons aussi une g\'en\'eralisation naturelle de la condition de contr\^ole g\'eom\'etrique, pour un amortissement seulement $L^\infty$. Cette g\'en\'eralisation est toujours n\'ecessaire pour la stabilisation uniforme (sur toute vari\'et\'e compacte r\'eguli\`ere), et nous d\'emontrons dans cet article qu'elle est suffisante dans notre cas particulier du tore $\mathbb{T}^2$ (et pour nos fonctions d'amortissement particuli\`eres).
\end{abstract} 

\ \vskip -1cm \hrule \vskip 1cm 
 \maketitle
{ \textwidth=4cm \hrule}
 \section{Notations and main results}
 Let $(M,g)$ be a (smooth) compact Riemanian manifold endowed with the metric $g$, $ \Delta_g$ the Laplace operator on functions on $M$ and for $a\in L^\infty(M)$,  let us consider the damped wave (or Klein-Gordon) equation 
 \begin{equation}\label{eq.amortie}
 (\partial_t ^2 - \Delta + a(x) \partial_ t  +m) u =0, \quad (u\mid_{t=0}, \partial_t u \mid_{t=0} ) = (u_0, u_1) \in (H^1\times L^2) (M),
 \end{equation}
 where $0 \leq m\in L^\infty(M)$. If $a\geq 0$ a.e. it is well known that the energy 
 \begin{equation}\label{energy}
  E_m(u) (t) = \int_M( |\nabla_g u|_g ^2+ |\partial_t {u}|^2+ m |u|^2) d{\rm vol}_g
  \end{equation}
 is decaying and satisfies 
 $$ E_m(u) (t) = E_m(u)(0)- \int_0 ^t \int _M 2 a(x) |\partial_tu |^2 d{\rm vol}_g .$$ 
 We shall say that the {\em uniform stabilisation} holds for the damping $a$  if one of the following equivalent properties holds (see appendix~\ref{app:C} for the equivalence).
 \begin{enumerate}
 \item There exists a rate $f(t)$ such that $\lim_{t\rightarrow + \infty } f(t) =0$ and for any $(u_0, u_1) \in (H^1\times L^2) (M)$,
 $$ E_m(u) (t) \leq f(t) E_m(u) (0).$$
 \item There exists $C,c>0$ such that  for any $(u_0, u_1) \in (H^1\times L^2) (M)$,
 $$ E_m(u) (t) \leq Ce^{-ct} E_m(u) (0).$$
 \item There exists $T>0$ and $c>0$ such that  for any $(u_0, u_1) \in (H^1\times L^2) (M)$, if $u$ is the solution to the damped wave equation~\eqref{eq.amortie}, then 
 $$ E_m(u) (0) \leq C \int_0 ^T \int _M 2 a(x) |\partial_tu |^2 d{\rm vol}_g.$$
 \item There exists $T>0$ and $c>0$ such that  for any $(u_0, u_1) \in (H^1\times L^2) (M)$, if $u$ is the solution to the undamped wave equation
 \begin{equation}\label{eq.nonamortie}
 (\partial_t ^2 - \Delta+m ) u =0, \quad (u\mid_{t=0}, \partial_t u \mid_{t=0} ) = (u_0, u_1) \in (H^1\times L^2) (M)
 \end{equation}
 then 
 $$ E_m(u) (0) \leq C \int_0 ^T \int _M 2 a(x) |\partial_tu |^2 d{\rm vol}_g.$$
 \end{enumerate}
 
The following result is classical (see the works by Rauch-Taylor~\cite{RaTa74, RaTa75}, Babich-Popov~\cite{BaPo81}, Babich-Ulin~\cite{BaUl81}, Ralston~\cite{Ra82},  Bardos-Lebeau-Rauch~\cite{BaLeRA92}, Burq-G\'erard~\cite{BuGe96}, Lebeau~\cite{Le96}, Koch-Tataru~\cite{KoTa95},  Sj\"ostrand~\cite{Sj00}, Hitrik~\cite{Hi04} )
\begin{theo}[Bardos-Lebeau-Rauch~\cite{BaLeRA92}, Burq-G\'erard~\cite{BuGe96}]
Let $m\geq 0$. Assume that the damping $a$ is continuous. For $\rho_0= (x_0, \xi_0) \in S^*M$ denote by $\gamma_{\rho_0}(s) $ the geodesic starting from $x_0$ in (co)-direction $\xi_0$. Then the damping $a$ stabilizes uniformly the wave equation iff the following geometric condition is satisfied 
\begin{equation*}\tag{GCC}\label{gcc}
\exists T, c>0; \inf_{\rho_0 \in S^*M} \int_0^T a( \gamma_{\rho_0} (s){)} ds \geq c.
\end{equation*}
\end{theo}
When the damping $a$ is no more continuous but merely $L^\infty$, we can prove
\begin{theo}\label{strong-weak}
Assume that $a \in L^\infty (M)$. Then the following strong geometric condition 
\begin{equation*} \tag{SGCC}\label{sgcc}
\exists T, c>0; \forall \rho_0 \in S^*M, \exists s\in (0, T), \exists \delta >0; a\geq c \text{ a.e. on } B( \gamma_{\rho_0} (s), \delta).
\end{equation*} 
is {\em sufficient} for uniform stabilisation, and  the following weak geometric condition 
\begin{equation*} \tag{WGCC}\label{wgcc}
\exists T >0; \forall \rho_0 \in S^*M, \exists s\in (0, T); \gamma_{\rho_0} (s)\in \text{supp} (a) 
\end{equation*} where $\text{supp}(a)$ is the support (in the distributional sense) of  $a$, is {\em necessary} for uniform stabilisation. 
\end{theo}
Though the question appears to be very natural, until the present work, the only known case in between was essentially an example of Lebeau~\cite[pp 15--16]{Le92-1} (from an idea of J. Rauch) where $M= \mathbb{S}^d$ and $a$ is the characteristic function of the half-sphere (notice however some refinements of~\eqref{wgcc} in~\cite{HuPrTr17, HuPrTr17-1}).  In this case of the half sphere, uniform stabilisation holds (see Zhu~\cite{Zh15} for a more detailled proof and a generalization of this result). 
\begin{theo}[Lebeau, \cite{Le92-1}]
On the $d$-dimensional sphere, 
$$\mathbb{S}^d= \{ x= (x_0, {\dots },  x_d) \subset \mathbb{R}^{d+1}; \| x\| =1\},$$ uniform stabilisation holds for the characteristic function of the half sphere $\mathbb{S}^d_+= \{ x= (x_0, {\dots },  x_d) \subset \mathbb{R}^{d+1}; \| x\| =1, x_0 >0\}$.
\end{theo}
\begin{rema} Notice that in this case, all the geodesics enter the interior of the support of $a$, and hence fulfill the~\eqref{sgcc} requirements, except the family of geodesics included in the boundary of the support of $a$, the $d-1$ dimensional sphere, 
$$\partial \mathbb{S}^d_+= \{ x= (x_0,{\dots },  x_d) \subset \mathbb{R}^{d+1}; \| x\| =1, x_0=0\}.$$
\end{rema}

 When the manifold is a two dimensional torus (rational or irrational)and the damping $a$ is a linear combination of characteristic functions of polygons, i.e.
there exists $N$, $R_j, j=1, \dots N$ (disjoint and non necessarily vertical) polygons and $0<a_j,j=1, \dots , N$ such that
   \begin{equation}\label{somme}
    a(x) = \sum_{j=1}^N a_j 1_{x\in R_j},
    \end{equation}
  { We can state  another natural simple geometric condition. 
  Let us endow the torus with an orientation (i.e. we see the torus as a surface in $\mathbb{R}^2$ and define at each point a normal vector $n(x)$). For any initial point $x_0$ and any norm-$1$ tangent vector $X_0$, let $\gamma$ be the geodesic starting from $x_0$ in direction $X_0$ and parametrized by arc length. Let $\nu(\gamma(s))$ be the unique vector normal to $\gamma$ in the torus and such that $(n(\gamma(s)) , \dot {\gamma}(s), \nu(\gamma(s)))$ is a direct orthonormal frame. By convention, we shall say that $\nu$ points to the left of the geodesic (and $- \nu$ to the right).}
   
   \begin{assu}\label{geom}
   Assume that the manifold is a two dimensional torus $\mathbb{T}^2= \mathbb{R}^2/ A\mathbb{Z} \times B\mathbb{Z}$, $a;B>0$. Assume that there exists $T>0$ such that all geodesics (straight lines) of length $T$ either encounters the interior of one of the polygons  or follows for some time one of the sides of a polygon $R_{j_1}$ {\em on the left} and for some time one of the sides of a   polygon $R_{j_2}$ (possibly the same) {\em on the right}. 
   \end{assu}
   Our main result is the following
\begin{theo}\label{theorem.1}
The damping $a$ stabilizes uniformly the wave equation if and only if Assumption~\ref{geom} is satisfied. 
   \end{theo}
   \begin{coro}
 Stabilisation holds for the examples 1.a and 1.d  but not for examples 1.b, 1.c and 1.e of figure~\ref{fig.1}
  \end{coro}
      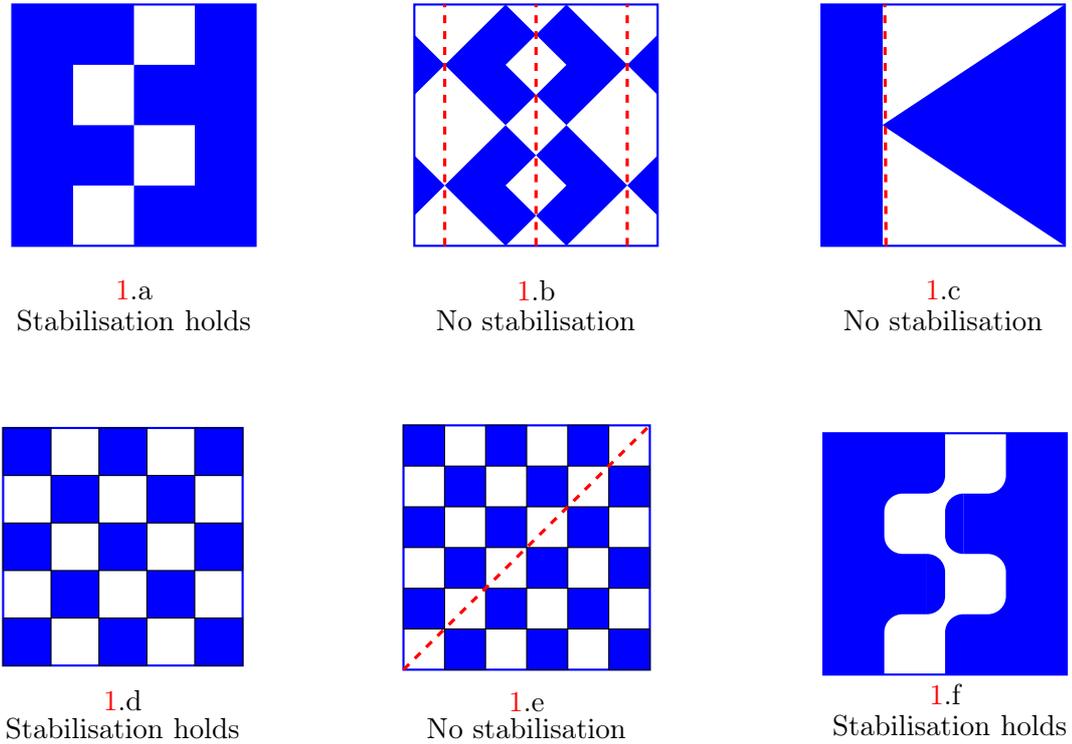
\begin{figure}[h]\label{fig.1}
 \begin{center}
\begin{tikzpicture}[scale=1.6]
\fill[blue] (-1,-1) -- (-1,1) -- (0,1) -- (0,0.5) -- (-0.5,0.5)-- (-0.5,0)-- (0,0) -- (0,-0.5)-- (-0.5,-0.5)-- (-0.5,-1)-- cycle;
\fill[blue] (1,-1) -- (1,1) -- (0.5,1) -- (0.5,0.5) -- (0,0.5) -- (0,0) -- (0.5,0)-- (0.5,-0.5) -- (0,-0.5)-- (0,-1) --(-0.5,-1)-- cycle;
\draw[thick, blue] (-1,-1) -- (-1,1) -- (1,1)--(1,-1)--cycle;
\draw (0, -1.2) node[below] {\protect{\ref{fig.1}}.a};
\draw (0, -1.45) node[below] {Stabilisation holds};
\end{tikzpicture}\hskip 2cm 
\begin{tikzpicture}[scale=1.6]
\fill[blue] (-.75,.5) -- (-.25,1) -- (0,.75) -- (-.25,0.5) -- (0,.25) -- (-.25,0) -- (0,-.25) -- (-.25,-.5)-- (0,-.75) --(-.25,-1)--(-.75,-.5)--(-.25,0)-- cycle;
\fill[blue] (.75,.5) -- (.25,1) -- (0,.75) -- (.25,0.5) -- (0,.25) -- (.25,0) -- (0,-.25) -- (.25,-.5)-- (0,-.75) --(.25,-1)--(.75,-.5)--(.25,0)-- cycle;
\fill[blue] (-.75,.5) -- (-1,.75) -- (-1,.25) -- cycle; 
\fill[blue] (.75,.5) -- (1,.75) -- (1,.25) -- cycle; 
\fill[blue] (-.75,-.5) -- (-1,-.75) -- (-1,-.25) -- cycle; 
\fill[blue] (.75,-.5) -- (1,-.75) -- (1,-.25) -- cycle; 
\draw[thick, blue] (-1,-1) -- (-1,1) -- (1,1)--(1,-1)--cycle;
\draw (0, -1.2) node[below] {\protect{\ref{fig.1}}.b};
\draw[very thick, red, dashed](-.75,-1) -- (-.75,1);
\draw[very thick, red, dashed](.75,-1) -- (.75,1);
\draw[very thick, red, dashed](0,-1) -- (0,1);
\draw (0, -1.45) node[below] {No stabilisation};
\end{tikzpicture}
\hskip 2cm 
\begin{tikzpicture}[scale=1.6]
\fill[blue] (-1,1) -- (-.5,1) -- (-.5,-1) -- (-1,-1) --  cycle;
\fill[blue] (1,1) -- (1,-1) -- (-.5,0) -- cycle;
\draw[thick, blue] (-1,-1) -- (-1,1) -- (1,1)--(1,-1)--cycle;
\draw[very thick, red, dashed](-.47,-1) -- (-.48,1);
\draw (0, -1.2) node[below] {\protect{\ref{fig.1}}.c};
\draw (0, -1.45) node[below] {No stabilisation};
\end{tikzpicture}
\end{center}
\vskip 1cm 
 \begin{center}
\begin{tikzpicture}[scale=.63]
\draw[thick, blue] (-3,-2) -- (2,-2) -- (2,3)--(-3,3)--cycle;
\drawcarre{-3}{2}{blue};
\drawcarre{-1}{2}{blue};
\drawcarre{-2}{1}{blue};
\drawcarre{0}{1}{blue};
\drawcarre{1}{2}{blue};
\drawcarre{-3}{0}{blue};
\drawcarre{-1}{0}{blue};
\drawcarre{1}{0}{blue};
\drawcarre{-2}{-1}{blue};
\drawcarre{0}{-1}{blue};
\drawcarre{-3}{-2}{blue};
\drawcarre{-1}{-2}{blue};
\drawcarre{1}{-2}{blue};
\draw (-.5, -2.3) node[below] {\protect{\ref{fig.1}}.d};
\draw (-.5, -2.9) node[below] {Stabilisation holds};
\end{tikzpicture}\hskip 2cm 
\begin{tikzpicture}[scale=.54]
\draw[thick, blue] (-3,-3) -- (3,-3) -- (3,3)--(-3,3)--cycle;
\drawcarre{-3}{2}{blue};
\drawcarre{-1}{2}{blue};
\drawcarre{1}{2}{blue};
\drawcarre{-2}{1}{blue};
\drawcarre{0}{1}{blue};
\drawcarre{2}{1}{blue};
\drawcarre{-3}{0}{blue};
\drawcarre{-1}{0}{blue};
\drawcarre{1}{0}{blue};
\drawcarre{-2}{-1}{blue};
\drawcarre{0}{-1}{blue};
\drawcarre{2}{-1}{blue};
\drawcarre{-3}{-2}{blue};
\drawcarre{-1}{-2}{blue};
\drawcarre{1}{-2}{blue};
\drawcarre{-2}{-3}{blue};
\drawcarre{0}{-3}{blue};
\drawcarre{2}{-3}{blue};
\draw[very thick, red, dashed](-3,-3) -- (3,3);
\draw (0, -3.3) node[below] {\protect{\ref{fig.1}}.e};
\draw (0, -3.95) node[below] {No stabilisation};
\end{tikzpicture} \hskip 2cm 
\raise 0.05cm \hbox{\begin{tikzpicture}[scale=1.6]
\fill[blue] (-1,-1) -- (-1,1) -- (0,1) -- (0,0.65) arc(0:-90:0.15) -- (-0.35,0.5) arc(90:180:0.15)--( -0.5, 0.15) arc(180:270:0.15)--(-0.15,0)-- (-0.15,-0.5)-- (-0.35,-0.5) arc(90:180:0.15) -- (-0.5,-1)-- cycle;
\fill[blue] (-0.15,0) arc(90:0:0.15) --(0,-0.35) arc(0:-90:0.15) -- cycle;
\fill[blue] (1,-1) -- (1,1) -- (0.5,1) -- (0.5,0.65) arc(0:-90:0.15)-- (0.15,0.5) -- (0.15,0) -- (0.35,0) arc(90:0:0.15)-- (0.5,-0.35) arc(0:-90:0.15) -- (0.15,-0.5) arc(90:180:0.15)-- (0,-1) --(-0.5,-1)-- cycle;
\fill[blue] (0.15,0.5) arc(90:180:0.15) --(0,0.15) arc(180:270:0.15) -- cycle;
\draw[thick, blue] (-1,-1) -- (-1,1) -- (1,1)--(1,-1)--cycle;
\draw (0, -1) node[below] {\protect{\ref{fig.1}}.f};
\draw (0, -1.25) node[below] {{ Stabilisation holds}};
\end{tikzpicture}}
\end{center}
\caption{Checkerboards: the damping 
$a$ is equal to $1$ in the blue region, $0$ elsewhere. For all these examples~\eqref{wgcc} is satisfied but not ~\eqref{sgcc}. The red dashed lines are geodesics which violate Assumption~\ref{geom}}
\end{figure}
{\begin{rema}
In assumption~\ref{geom}, as soon as the damping is non trivial (i.e. we have at least one polygon), all non closed geodesics will enter the interior of this polygon (because any non closed geodesics is dense in the torus). As a consequence, the second part of the assumption has to be checked only for closed geodesics. Actually, closed geodesics corresponding to directions $(\xi, \eta) = \frac{(p,q) }{ \sqrt{ p^2 + q^2}}$, $p \wedge q =1$, will also enter the polygon as soon as $p^2 + q^2$ is large enough. As a consequence, the second part of Assumption~\ref{geom} has to be checked only for a {\em finite number} of closed geodesics
\end{rema}
\begin{rema}
As pointed out by a referee, our proof actually gives a sufficient condition for stabilisation in a more general setting where the $R_j$ need not be polygons, but are open subsets and we assume that all but a finite number of closed geodesics are damped (in the sense that they enter the interior of one of the $R_{j}$'s) and the remaining closed geodesics satisfy the left/right property on intervals of positive measure (which implies that the boundary of the  open sets $R_{j_1}$ and $R_{j_2}$have some flat parts). As a consequence, stabilisation holds for  Figure \ref{fig.1}.f.
\end{rema} }

\begin{rema}
Stabilisation implies that exact controlability holds for some finite $T>0$. However our proof relies on a contradiction argument and resolvent estimates. It gives no geometric interpretation for this controlability time. This is this contradiction argument which allows us on tori to avoid a particularly delicate regime at the edge of the uncertainty principle (see Section~\ref{sec.2.2}). Giving a geometric interpretation of the time necessary for control would require dealing with this regime (see~\cite{Bu19}).
\end{rema}
 The plan of the paper is the following: In Section~\ref{sec.2}, we focus on the model case of the left checkerboard in Figure~\ref{fig.1}. We first reduce the question of uniform stabilisation to the proof of an observation estimate for high frequency solutions of Helmholtz equations.  We proceed by contradiction and construct good quasi-modes, for the study of which we perform a micro-localization which shows that the only obstruction is the vertical geodesic in the middle of the board. Then we prove a non concentration estimate which shows that solutions of Helmholtz equations (quasi-modes) cannot concentrate too fast on this trajectory. This is essentially the only point in the proof which is specific to the torus and it relies on the special geometric structure of the torus which was previously used in the context of Schr\"odinger equations~\cite{BuZw03, BuZw03-1, Ma10, BuZw12, AnMa14, BBZ13} and also for wave equations~\cite{BuHi05, AL14}.  Finally, by means of a second micro-localization with respect to this vertical geodesic, we obtain a contradiction.
In Section~\ref{sec.3}, we show how the general case can be reduced to this model case. Finally, in the last section we introduce a generalized version of~\eqref{gcc} that makes sense for $a\in L^\infty$ and which is equivalent to Assumption~\ref{geom} in our particular case. We prove that this generalized geometric control condition is always necessary (on any Riemannian manifold and for any damping $0\leq a\in L^\infty$) and we conjecture that it is always sufficient. For the convenience of the reader, we gathered in an appendix a few quite classical resultsabout the link between resolvent estimates and stabilisation.

The {\em second micro-localization} procedure has a well established history starting with the works by Laurent~\cite{La79, La85}, Kashiwara-Kawai~\cite{KaKa80},  Sj\"ostrand~\cite{Sj82}, Lebeau~\cite{Le85} in the analytic context, (see also Bony-Lerner~\cite{BoLe89} in the $C^\infty$ framework and Sj\"ostrand-Zworski~\cite{SjZw99} in the semi-classical setting) and in the framework of defect measures by Fermanian--Kammerer~\cite{FK00}, Miller~\cite{Mi96, Mi97, Mi00}, Nier~\cite{Ni96}, Fermanian--Kammerer-G\'erard~\cite{FKGe02, FKGe03, FKGe04}. Notice that most of these previous works in the framework of measures dealt with lagrangian or involutive  sub-manifolds, and it is worth comparing our contribution with these previous works, in particular~\cite{Ni96, AnMa14}. Here we are interested in the wave equation while the authors in~\cite{Ni96, AnMa14} were interested in the Schr\"odinger equation, and (compared to~\cite{AnMa14}) we are dealing with  worse quasi-modes ($o(h)$ instead of $o(h^2)$).  Another difference is that we perform a second microlocalization along a symplectic submanifold (namely $\{ (x=0,y, \xi =0, \eta) \in T^* \mathbb{T}^2\}$), while they consider an isotropic submanifold $\{ x=0\}$ in ~\cite{Ni96} or $\{ (x' , x'', \xi'  =0,\xi'')  \in T^* \mathbb{T}^d\}$ in~\cite{AnMa14}. {An exception is the note by Fermanian--Kammerer~\cite{FK05}, to which our construction is very close.} On the other hand, a feature shared by the present work and~\cite{Ni96, AnMa14} is that in all cases the analysis requires to work at the edges of uncertainty principle and use refinements of some exotic Weyl-H\"ormander  classes ($S^{1,1}$ in~\cite{Ni96}, $S^{0,0}$ in ~\cite{AnMa14} and $ S^{1/2, 1/2}$ in the present work), see~\cite{Ho85} and L{\'e}autaud-Lerner~\cite{LeLe14} for related work. {Another worthwhile comparison is with the series of works by Burq-Hitrik~\cite{BuHi05} and Anantharaman-Leautaud~\cite{AL14} on the damped wave equation on the torus when the control domain is arbitrary (in this case (WGCC) is in general not satisfied). However, though both works use some kind of second microlocalisation and deal with the wave equation, in~\cite{BuHi05, AL14} the approaches use Schr\"odinger equations methods (strong quasi-modes) transposed to get wave equations result and consequently leads to  much weaker results (polynomial decay v.s. exponential decay) under much weaker assumptions (arbitrary open sets).}

 \medskip\noindent \textbf{Acknowledgements.}
This research was partially supported by 
Agence Nationale de la Recherche through project ANA\'E ANR-13-BS01-0010-03
(NB \& PG)
{
\section{First micro-localization, proof of Theorem~\ref{strong-weak}}\label{sec.2.1}
In this section we work on an arbitrary compact manifold $M$ with an arbitrary damping function $a\in L^\infty(M)$ and outline the classical propagation arguments  which show that (SGCC) is sufficient for stabilisation while (WGCC) is necessary.
 Let us assume (SGCC) holds. According to Proposition~\ref{prop.B.2}, we need to prove~\eqref{obs}
 \begin{equation*}\tag{\protect{\ref{obs}}}
   \begin{gathered}
 \exists h_0>0; \forall 0<h<h_0, \forall  (u,f) \in H^2(M)\times L^2(M), (h^2 \Delta + 1) u =f,  \\
 \| u\|_{L^2(M)} \leq C \bigl( \| a^{1/2} u\|_{L^2} + \frac {1}{h} \| f\|_{L^2}\Bigr).
 \end{gathered}
 \end{equation*}
  To prove this estimate we argue by contradiction and obtain sequences $(h_n)\rightarrow 0$, and $(u_n, f_n)$ such that 
 $$(h_n^2\Delta +1) u_n = f_n, \quad \| u_n \|_{L^2}=1, \| a^{1/2} u_n \|_{L^2} = o(1)_{n\rightarrow + \infty}, \| f_n \|_{L^2} = o(h_n)_{n\rightarrow + \infty}.$$
 Extracting a subsequence, we can assume that the sequence $(u_n)$ has a semi-classical measure $\nu$ on $T^*\mathbb{T}^2$.
 For $q\in C^\infty_0 ( T^*M)$, we define in $Op_h (q)$ by the following procedure. Using partition of unity, we can assume that $q$ is supported in a local chart. Then, in this chart, we define
  \begin{equation}
 Op_h (q) (u) = \frac 1 {(2\pi h)^d} \int e^{\frac i h (x-y) \cdot \xi} q(x, \xi ) \zeta (y) u(y) dy d\xi,
 \end{equation}
 where $\zeta=1$ in a neighborhoud of the support of $q$ (remark that modulo smoothing $O(h^\infty)$ errors, this quantisation does not depend on the choice of the cut-off $\zeta$). 
 Then a semi-classical measure for the sequence $(u_n)$ satisfies
 $$ \lim_{n\rightarrow+ \infty} \Bigl(Op_{h_n}(q)u_n, u_n \Bigr)_{L^2(M)} = \langle \nu, q \rangle.
 $$
   In our case, it is supported in the c{h}aracteristic set 
 $$ \{ (X, \Xi) \in S^*\mathbb{T}^2; \| \Xi \|^2 =1\}.$$ 
 Furthermore, this measure has total mass $1$ and is invariant by the bic{h}aracteristic flow:
 $$ \Xi \cdot \nabla_X \nu =0.$$
 We refer to~\cite[Section 3]{Bu02} for a proof of these results in a very similar context.
 Let 
 $$S = \{ x\in M; \exists \delta >0, c>0; a \geq c \text{ on } B(x, \delta)\}.$$
Since $\| a^{1/2}u_n \| _{L^2} = o(1)_{n\rightarrow + \infty}$ we get that the measure $\nu$ vanishes above every point in $S$. The assumption (SGCC) ensures that every bicharacteristic contains at least one point in $S$. Hence $\nu$ is identically $0$ which contradicts the fact that it has total mass $1$! 
    }
    
{ Let us now assume that (WGCC) is not satisfied and prove that stabilisation does not hold. We actually prove the more precise result, which according the equivalence of properties (1) to (4) above (see Appendix~\ref{app:C}) implies that stabilisation does not hold
\begin{prop}
Let $T>0$. Consider a geodesic $\gamma$ of length $T$ which does not encounter the support of the damping function $a$.  Then there exists a sequence $(u_n)$ of solution to the wave equation~\eqref{eq.nonamortie} which satisfies 
\begin{equation} \lim_{n\rightarrow + \infty}E_m(u_n) =1, \qquad \lim_{n\rightarrow + \infty} \int_0^T \int_{M} a(x) |\partial_t u|^2 (t,x) dx dt =0.
\end{equation}
\end{prop}
}
{
First by compactness, there exists $\delta >0$ such that 
$$ \text{dist} ( \gamma ([- \delta, T+\delta]), \text{supp} (a) ) \geq \delta.
$$ 
Then, according to Proposition~\ref{ralston}, there exists  a sequence of approximate solutions $(v_n)$ to the wave equation (with $m=0$)  which is exponentially localised on the geodesic $\gamma$ and satisfies~\eqref{approximate} and~\eqref{upper}. 
From~\eqref{upper}, we deduce that 
\begin{equation}\label{decaybis}
 \| v_n\|_{L^2(M)}  = O(h_n),
 \end{equation}
and 
from the exponential localisation on the geodesic $\gamma$ (and the $\delta$ separation with the support of $a$)
\begin{equation}
\label{decroissance}
 \int_0^T \int_M a(x) |\partial_t v_n |^2 (t,x) dx dt = O( e^{-c\frac{ \delta^2} { h_n}})
\end{equation} 
uniformly with respect to $t\in [-1, T+1]$. The solution $u_n$ to the wave equation (with $m$)~\eqref{eq.nonamortie} with the same initial data satisfies
$$ (\partial_t^2 - \Delta + m) (u_n - v_n) = - m v_n, \qquad, (u_n - v_n)\mid_{t=0} =0, \partial (u_n - v_n)\mid_{t=0} =0.$$
As a conseqnence from Duhamel formula and~\eqref{decaybis}, 
$$ \sup_{t\in [0,T]} \| u_n - v_n \|_{H^1(M)} ^2 + \| \partial_t u_n - \partial _t v_n \|_{L^2(M)} ^2 = O(h_n ^2).
$$ 
This implies according to~\eqref{decroissance}
$$ E_m (u_n) =1 + O(h_n ^2), \qquad  \int_0^T \int_M a(x) |\partial_t v_n |^2 (t,x)dx dt = O(h_n^2).
$$
}
 
 \section{The model case of a checkerboard}\label{sec.2}
 
In this section we prove Theorem~\ref{theorem.1} for  the following model on the two dimensional torus $\mathbb{T}^2= \mathbb{R}^2 / (2\mathbb{Z})^2$. We shall later microlocally reduce the general case to this model.
 \begin{figure}[h]
 \begin{center}
\begin{tikzpicture}[scale=1.6]
\fill[blue] (-1,-1) -- (-1,1) -- (0,1) -- (0,0.5) -- (-0.5,0.5)-- (-0.5, 0)-- (0,0)-- (0, -0.5)-- (-0.5,-0.5) --(-0.5,-1)-- cycle;
\fill[blue] (1,-1) -- (1,1) -- (0.5,1) -- (0.5,0.5)-- (0,0.5)-- (0,0)--(0.5,0) -- (0.5,-0.5) -- (0,-0.5) -- (0,-1) -- cycle;
\draw[thick, blue] (-1,-1) -- (-1,1) -- (1,1)--(1,-1)--cycle;
\draw[thick, ->] (-1.4,0) -- (1.4,0);
\draw[thick, ->] (0,-1.2) -- (0,1.4);
\draw (1.1,0) node[below] {$1$};
\draw (1.4,0) node[below] {$x$};
\draw (0,1.1) node[left] {$1$};
\draw (0,1.4) node[left] {$y$};
\end{tikzpicture}
\end{center}
\caption{The checkerboard: a microlocal model 
where the damping 
$a$ is equal to $1$ in the blue region, $0$ elsewhere}\label{fig.4}

\end{figure}
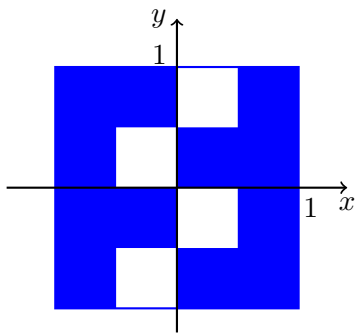
According to the results in Section~\ref{sec.2.1},  since the only two bicharacteristic{s} which do not enter the interior of the set where $a=1$ are 
$$\{ (x=0,\xi=0, \eta = \pm 1)\},$$
we know that {$\nu$} is supported on the union of these two bicharacteristics.

\subsection{A priori non concentration estimate}\label{sec.2.2}
In this section we show that $(u_n)$ cannot concentrate on too small neighbourhoods around $\{ x=0\}$. This is the key (and only) point where we use the particular structure of the torus as a product manifold.
 
  Let us recall that $\| (h_n^2 \Delta + 1) u_n\|_{L^2} = o(h_n)$.  Define 
  
  \begin{equation}
  \label{epsi}
   \epsilon (h_n) = \max \left ( h_n^{{1/6}}, \Bigl( \| (h_n^2\Delta +1) u_n\| / h_n \Bigr) ^{1/6}\right ),
  \end{equation}
   so that 
  \begin{equation}\label{borne}
   h_n^{-1} \epsilon^{-6} (h_n)\|  (h_n^2 \Delta+1) u_n\|_{L^2}\leq 1, \qquad  \lim_{n\rightarrow +\infty } \epsilon (h_n) =0.
   \end{equation}
    The purpose of this section is to prove the following non concentration result which is actually related to Kakeya-Nikodym bounds (see~\cite{So11, BlSo15, ChSoYa16}))
 \begin{prop}\label{aprioriprop}
 Assume that $\|u_n\|_{L^2} = \mathcal{O}(1)$, and ~\eqref{borne} holds. Then there exists $C>0$ such that 
 $$ \forall n \in \mathbb{N}, \| u_n\|_{L^2( \{ |x| \leq  h_n^{1/2} \epsilon ^{-2} (h_n) \} )} \leq C \epsilon^{1/2} (h_n). $$
 \end{prop}
The proposition follows from the following one dimensional propagation estimate (see~\cite{BuZu15} for related estimates)
\begin{prop}\label{rescalled}
There exists $C>0, h_0$ such for any $0<h<h_0$,  $1\leq \beta\leq h^{-\frac 1 2}$, and any $(u,f)\in H^2 \times L^2$ solutions of 
$$ (h^2(\partial_x^2 + \partial_y^2)+1 ) u =f, $$
we have 
{\begin{multline}\label{eq.rescalled}
\| u\|_{L^\infty( \{ |x| \leq\beta h^{\frac 1 2}\}; L^2_y )}
\\
\leq C\beta^{- \frac 1 2} h^{-\frac 1 4}  \Bigl( \| u\|_{L^2( \{\beta h^{\frac 1 2} \leq  |x| \leq2 \beta h^{\frac 1 2}\}; L^2_y )} + h^{-1} \beta^2  \| f\|_{L^2( \{ |x| \leq2 \beta h^{\frac 1 2}\}; L^2_y )} \Bigr ).
\end{multline}}
\end{prop} 
Let us first show that Proposition~\ref{aprioriprop} follows from Proposition~\ref{rescalled}. Indeed, choosing $\beta= \epsilon^{-3}(h)$, H\"older's inequality gives 
\begin{multline}
\| u\|_{L^2( \{ |x| \leq h^{\frac 1 2}\epsilon^{-2}(h)\} )}\leq 
h^{\frac 1 4} \epsilon^{-1} (h)  \| u\|_{L^\infty( \{ |x| \leq h^{\frac 1 2}\epsilon^{-3}(h\} )}\\
\leq C \epsilon^{\frac 1 2}(h)\Bigl( \| u\|_{L^2( \{ h^{\frac 1 2}\epsilon^{-3}(h) \leq  |x| \leq2 h^{\frac 1 2}\epsilon^{-3}(h)\} )} + h^{-1} \epsilon^{-6}(h)  \| f\|_{L^2( \{ |x| \leq2 \beta h^{\frac 1 2}\} )}\Bigr)\\
\leq C \epsilon ^{\frac 1 2}(h) \Bigl(\|u\|_{L^2}+ h^{-1} \epsilon^{-6} (h)\| f\|_{L^2}\Bigr)
\leq 2C \epsilon^{\frac 1 2} (h),
\end{multline}
where in the last inequality we used~\eqref{borne}. 

Now we can prove Proposition~\ref{rescalled}. Denote by $v$ (resp. $g$) the Fourier transform of $u$ (resp. $f$). For fixed $x$,
$$\| v(x,\cdot)\|_{L^2_\eta}= 2\pi \| u(x, \cdot)\|_{L^2_y}.$$
{ We deduce that~\eqref{eq.rescalled} is equivalent to
\begin{equation}\label{eq.rescalledbis}
\| v\|_{L^\infty( \{ |x| \leq\beta h^{\frac 1 2}\}; L^2_\eta )}
\leq C\beta^{- \frac 1 2} h^{-\frac 1 4}  \Bigl( \| v\|_{L^2( \{\beta h^{\frac 1 2} \leq  |x| \leq2 \beta h^{\frac 1 2}\}; L^2_\eta )} + h^{-1} \beta^2  \| f\|_{L^2( \{ |x| \leq2 \beta h^{\frac 1 2}\}; L^2_\eta )} \Bigr ).
\end{equation}
Now, by Minkovski inequality,
$$ \| v|\|_{L^\infty_x; L^2_\eta} \leq \| v\| _{L^2_\eta; L^\infty_x}$$}
and we deduce that~\eqref{eq.rescalledbis} is implied by 

\begin{prop}\label{rescalledter}
There exists $C>0, h_0$ such for any $0<h<h_0$,  $\eta\in \mathbb{R}$ $1\leq \beta\leq h^{-\frac 1 2}$, and any $(v,g)$ solutions of 
$$(h\beta^{-2} \partial_z^2 + 1- h^2\eta^2)2)v=g,$$
\begin{equation}\label{eq.rescalledter}
\| u\|_{L^\infty( \{ |x| \leq\beta h^{\frac 1 2}\})}
\leq C\beta^{- \frac 1 2} h^{-\frac 1 4}  \Bigl( \| u\|_{L^2( \{\beta h^{\frac 1 2} \leq  |x| \leq2 \beta h^{\frac 1 2}\})} + h^{-1} \beta^2  \| f\|_{L^2( \{ |x| \leq2 \beta h^{\frac 1 2}\})} \Bigr ).
\end{equation}
\end{prop} 

We change variables $x= \beta h^{\frac 1 2} z$, and it is enough to prove, for solutions of 
$$(h\beta^{-2} \partial_z^2 + 1- h^2 \eta^2)v=g,$$
\begin{equation}
\| v\|_{L^\infty( \{ |z|\leq 1\} )}\\
\leq C  \Bigl( \| v\|_{L^2( \{1\leq  |z| \leq 2 \} )} + h^{-1} \beta^2  \| {g}\|_{L^2( \{ |z| \leq 2 \} )} \Bigr ).
\end{equation}
Finally, this latter estimate follows (with $\tau= \beta^2 h^{-1} (1- h^2 \eta^2) $) from the following result which is generalization of ~\cite[Proposition 3.2]{BuZu15} (remark that taking benefit of the dimension $1$, we can replace the $L^2$ norm in the left of~\cite[Proposition 3.2, (3.3)]{BuZu15} by an $L^\infty$ norm).
\begin{lemm}
There exists $C>0$ such that, for any $\tau \in \mathbb{R}$ and any solution $(v,k)$ on $(-2,2)$ of 
$$ (\partial_z^2 + \tau)v=k,$$
then
$$\|v\|_{L^\infty(-1,1)}\leq C \Bigl (\| v\|_{L^2(\{ 1\leq |z|\leq 2\} )} + \frac 1 {\sqrt{ 1+ |\tau|}} \| k\|_{L^1(-2,2)}\Bigr),$$

\end{lemm} 
Let $\chi \in C^\infty_0 (-2,2)$ equal to $1$ on $(-1,1)$. Then $u= \chi v$ satisfies 
\begin{equation}\label{statio}
  (\partial_z^2 + \tau ) u = \chi {k}+ 2 \partial_z (\chi' v) - \chi '' v.
  \end{equation}
{W}e distinguish two regimes.\\
$\bullet$ Elliptic regime, $\tau \leq -1$. Then, multiplying by $u$ and integrating by parts gives
 \begin{multline}
 \| \partial_zu\|^2_{L^2(-2,2)}+ |\tau| \|u\|^2_{L^2(-2,2)} \\
 = {-}\Bigl( \chi {k}+ 2 \partial_z (\chi' v) - \chi '' v, u \Bigr)_{L^2} = {-}\Bigl( \chi {k} - \chi '' v, u \Bigr)_{L^2}{+} 2\Bigl( \chi' v, \partial_z u \Bigr)_{L^2},
 \end{multline}
 which implies 
 \begin{multline}
 \| \partial_zu \|^2_{L^2(-2,2)}+ |\tau | \| u\|^2_{L^2(-2,2)} \\
 \leq C \Bigl( \| {k}\|_{L^1(-2,2)}\| u\|_{L^\infty} + \| v\|_{L^2(\{1\leq |z|\leq 2\})} (\|{u}\|_{L^2(\{1\leq |z|\leq 2\})} + \| \partial _z u\|_ {L^2(-2,2)}{)} \Bigr), 
 \end{multline}
 and the one-dimensional Gagliardo-Nirenberg inequality 
 $$ \| u\|_{L^\infty} \leq C {\| \partial _zu\|_{L^2}}^{1/2} \| u\|_{L^2}^{1/2}
 $$ allows to conclude in this regime. \\
 $\bullet$ Hyperbolic regime, $\tau \geq -1$. Let $\sigma= \sqrt{\tau} \in \mathbb{R}^+ \cup i [0,1]$. 
 The solution of 
 ~\eqref{statio} is
 \begin{eqnarray*}
  u(x) &=& \int_{y=-2}^x e^{-i \sigma(x-y)} \int_{z=-2}^y e^{i\sigma(y-z)} g(z) dz dy \\
 &=&\int_{z=-2}^x g(z) \int_{y=z}^x e^{i\sigma(2y -x-z)} dy dz,
 \end{eqnarray*}
  where $g= \chi {k} - \chi '' v+ 2 \partial_z (\chi' v)= g_1 + \partial_z g_2.$
 Since, {for $x,z\in [-2,2]$, }
 $$\bigl | \int_{y=z}^x e^{i\sigma(2y -x-z)} dy\bigr| \leq \frac{ C} {1+ |\sigma|},$$
 the contribution of $g_1$ is uniformly bounded by 
 $$\frac C{1+ |\tau|} (\| \chi {k}\|_{L^1(-2,2)} +  \| {v} \|_{L^1({\{1\leq |z|\leq 2\})}}).$$
 Integrating by parts in the integral involving $\partial_z g_2$, we see that similarly, the contribution of $\partial_z g_2$ is bounded by 
 $$C \| \chi '  v\|_{L^1(-2,2)} .$$

 \subsection{Second micro-localization}In this section we develop the tools required to understand the concentration properties of our sequence $(u_n)$ on the symplectic sub-manifold $\{x=0, \xi=0\}$ of the phase space $T^* \mathbb{T}^2$. {The construction is very close to the one in Fermanian--Kammerer~\cite{FK05}.}
 \subsubsection{Symbols and operators}\label{sec.4.1}
 We define  $S^m$ the class of smooth functions of the variables $(X, \Xi, z, \zeta)\in \mathbb{R}^2 \times \mathbb{R}^2\times \mathbb{R}\times \mathbb{R}$ which have compact supports with respect to the $(X,\Xi)$ variables and are polyhomogeneous of degree $m$ with respect to the $(z, \zeta)$ variables, with limits in the radial direction 
 $$ \lim_{r \rightarrow + \infty} \frac 1{r^m} a\left (X, \Xi, \frac{ (rz, r\zeta )} {\| (z, \zeta)\|} \right )= \widetilde{a} \left (X, \Xi, \frac{ (z, \zeta )} {\| (z, \zeta)\|}\right ).$$
  When $m=0$, via the change of variables 
 $$(z, \zeta) \mapsto (\widetilde z, \widetilde \zeta)  = \frac {( z, \zeta)} {\sqrt{ 1+ |z|^2+ |\zeta|^2}} ,$$
 such functions are identified with smooth compactly supported functions on $\mathbb{R}^4_{(X, \Xi)} \times \overline{B(0,1)}_{\widetilde{z}, \widetilde{\zeta}}$, where $\overline{B(0,1)}$ denotes  the closed unit ball in $\mathbb{R}^2$.

Let $\epsilon(h)$ satisfying
$$ \lim_{h\rightarrow 0 } \epsilon (h) =0, \qquad \epsilon (h) \geq h^{1/2}.$$

 In order to perform the second micro-localization around the sub-manifold given by the equations 
 $x=0, \xi =0$, we define,  for $a\in S^m$,
 $$ \text{Op}_h (a) = a\left (x,y, hD_x, hD_y,  \frac{ \epsilon(h)}{h^{1/2}} x, \epsilon (h) h^{1/2} D_x\right ),$$
 {where $X=(x,y)$, $\Xi =(\xi ,\eta )$.}
 Notice that this quantification is the usual one~\cite{Ho85}, associated to the symbol $$a\left (x,y, \xi, \eta, \frac{ \epsilon(h)}{h^{1/2}} x, \epsilon (h) h^{1/2} \xi \right ) . $$ A simple calculation shows that since  $\epsilon({h}) \geq h^{1/2}$, {the  latter} symbol belongs to the class $S(\bigl( 1+ \epsilon^2 (h) h^{-1} {x}^2+\epsilon^2 (h) h {\xi}^2\bigr)^{m/2},g)$ of the Weyl-H\"ormander calculus \cite{Ho85} for the metric 
 \begin{multline}
 g=\frac{ \epsilon^2(h)}{h}\frac{ dx^2}{1+ \epsilon ^2(h)  h^{-1}x^2 + \epsilon ^2 (h) h \xi^2} + \epsilon^2(h)h\frac{ d\xi^2}{ 1+ \epsilon ^2(h)  h^{-1}x^2 + \epsilon ^2 (h) h \xi^2} \\
 + \frac{ dy^2} {1+y^2 + h^2\eta^2} + h^2 \frac{ d\eta^2} { 1+y^2 + h^2\eta^2}.\end{multline}
 
 As a consequence,  we deduce that the operator{s} such defined enjoy good properties and we have a good symbolic calculus, {namely} for all $a \in S^0$, the operator $\text{Op}(a)$ is bounded on $L^2( \mathbb{R}^2)$ uniformly with respect to $h$, and 
 $$ \forall a \in S^p, b\in S^q,  ab \in S^{p+q} \text{ and }\text{Op} (a) \text{Op} (b) = \text{Op}(a b) + \epsilon^2 (h) r,
 $$
 where $r\in \text{Op}(S^{p+q-1})$, and 
 $$ \forall a \in S^0, a\geq 0 \Rightarrow \exists C>0; \Re (\text{Op} (a)) \geq - C \epsilon ^2(h), \|\Im ( \text{Op} (a))\|\leq C \epsilon ^{{2}}(h) .$$

 \subsubsection{Definition of the second semi-classical measures}
 In this Section, we consider a sequence $(u_n)$ of functions on the two dimensional torus $\mathbb{T}^2$ such that 
 \begin{equation}\label{eq.osc}
  (h_n^2 \Delta + 1) u_n = \mathcal{O}(1)_{L^2},
  \end{equation}
 We identify $u_n$ with a periodic function on $\mathbb{R}^2$. Now, using the symbolic calculus properties in Section~\ref{sec.4.1},  {and in particular G\aa rding inequality and the $L^2$ boundedness of operators}, we can extract a subsequence (still denoted by  $(u_n)$) such that there exists a positive measure $\widetilde{\mu}$ on $T^* \mathbb{T}^2 \times \overline {N}$ --- $\overline {N}$ {denotes} the sphere compactification of $N= \mathbb{R}^2_{z,\zeta}$ --- such that, for any  symbol $a\in {S^0}$,
 $$ \lim_{n\rightarrow + \infty} \Bigl( \text{Op}_{h_n} (a) u_n, u_n \Bigr)_{L^2}= \langle \widetilde{\mu}, \widetilde{a}\rangle,$$
 where the continuous function, $\widetilde{a}$  on $T^* \mathbb{R}^2 \times \overline{N}$ is naturally defined in the interior by the value of the symbol $a$ and on the sphere at infinity by
 $$ \widetilde{a} (x,y,\xi, \eta, \widetilde{z}, \widetilde{\zeta})= \lim_{r\rightarrow + \infty} a(x,y,\xi, \eta, r  \widetilde{z}, r\widetilde{\zeta}),$$
 (which exists because $a$ is polyhomogeneous of degree $0$).
 The measure $\widetilde{\mu}$ is of course periodic, and hence defines naturally a measure $\mu$ on $T^* \mathbb{T}^2 \times \overline{N}$, and using~\eqref{eq.osc}, it is easy to see that there is no loss of mass at infinity in the $\Xi$ variable:
 \begin{equation}\label{mass}
  \mu (T^* \mathbb{T}^2 \times \overline{N})= \lim_{n\rightarrow + \infty } \| u_n\|^2_{L^2 ( \mathbb{T}^2)} .
  \end{equation}
 \subsubsection{Properties of the second semi-classical measure}
 In this section, we turn to the sequence constructed in {Section~\ref{sec.2.1}} and study refined properties of the second semi-classical measure constructed above, for the choice $\epsilon (h) $ given by~\eqref{epsi}. Notice that compared to~\eqref{eq.osc} the sequence considered here satisfies the stronger
 $$ (h_n ^2 \Delta + 1) u_n = o( h_n)_{L^2}.$$
 \begin{prop}\label{prop-propag}
 The measure $\mu$ satisfies the following properties.   
 {\begin{enumerate}
 \item \label{4.1} Assume only that $$ (h_n ^2 \Delta + 1) u_n = O(1)_{L^2}.$$ Then the measure $\mu $ has total mass $1= \ \| u_n\|_{L^2}^2  $ ($h_n$- oscillation)
 \item \label{4.2} Assume now that $$ (h_n ^2 \Delta + 1) u_n = o(h_n)_{L^2} $$ and $\| a u_n \|_{L^2} = o(1)$. Then, since the projection of the measure $\mu$ on the $(x,y,\xi, \eta)$ variables is the measure $\nu$ of {Section~\ref{sec.2.1}} which is invariant by the bicharacteristic flow, we get that the measure $\mu$  is supported on the set $$\{ (x,y, \xi, \eta, z, \zeta); x=0, \xi =0, \eta = \pm 1\} $$ 
 \item \label{4.4} Assume now that $$ (h_n ^2 \Delta + 1) u_n = O(h_n \epsilon ( h_n))_{L^2} $$ Then the measure $\mu $ is supported on the sphere at infinity in the $(z, \zeta)$ variables. 
  \item \label{4.5} Assume now that  $$ (h_n ^2 \Delta + 1) u_n = O(h_n \epsilon ( h_n))_{L^2}, \qquad \| 1_R u_n \|_{L^2} = o(1), $$ where $R$ is a polygon. Then the measure $\mu$ vanishes $2$-microlocally at each point of $\partial R$ on the side where the polygon $R$ lies. Namely in our geometry,  the measure $\mu$ vanishes $2$-microlocally on the right on $\{ x=0, y\in(0, \frac 1 2) \cup (-1, - \frac 1 2)\}$ and $2$-microlocally on the left on $\{ x=0, y\in(-\frac 1 2, 0) \cup (\frac 1 2, 1)\}$, {more precisely,}
  \begin{equation}\label{apriori}
  \begin{gathered}
   \mu (\{ (x,y,\xi, \eta, z,\zeta) ; x=0, y\in (0, 1/ 2)\cup(-1, - 1 /2), z >0\})=0\\
    \mu (\{ (x,y,\xi, \eta, z,\zeta) ; x=0, y\in (- 1/ 2, 0)\cup  (1/ 2, 1), z <0\}) =0
  \end{gathered}
  \end{equation}
  \item  According to point ~\ref{4.4} above, if we identify  the sphere at infinity in the $(z, \zeta)$ variables with $\mathbb{S}^1$ by means of the choice of variables $z= r\cos ( \theta), \zeta= r\sin (\theta), r\rightarrow + \infty$, the measure $\mu$ can be seen as a measure in $(x,y,\xi, \eta, \theta)$ variables, supported on $x=0, \xi=0, \eta = \pm 1$. 
In this  coordinate system, we have
    \begin{equation}\label{eqprop}
   (\eta \partial_y - \sin^2 (\theta) \partial_\theta )\mu =0.  
   \end{equation}
     \end{enumerate}}
   \end{prop}
   {\begin{rema}
   In Proposition~\ref{prop-propag}, the only point where we use crucially the particular geometry of the torus (Proposition~\ref{aprioriprop}) is  point~\ref{4.4}. For more general geometries, this point is no more true.  However, for the part on the sphere at infinity of the measure, we can still get an analog of ~\eqref{eqprop} for more general geometries,  involving the curvature of the surface along the geodesic (see~\cite{Bu19}). 
   \end{rema}
   }
   \begin{proof}
   The proof of point~\ref{4.1} follows from~\eqref{mass}. To prove point~\ref{4.2}, we just remark that the choice of test functions $a(x,\Xi, z, \zeta)= a(X, \Xi)$ shows that the direct image {$\pi_*(\mu)$} of $\mu$  by the map 
   $$\pi: (X, \Xi, z, \zeta) \mapsto (X,\Xi),$$
    is actualy the (first) semi-classical measure $\nu$ constructed in Section~\ref{sec.2.1}, and consequently, this property follows from {Section~\ref{sec.2.1}}. 
   To prove point~\ref{4.4}, we recall that  from Proposition~\ref{aprioriprop}, we have that for any $\chi\in C^\infty_0$, bounded by $1$ and supported in $(-A,A)$ 
 \begin{multline} \| \chi (h_n^{-1/2} \epsilon (h_n)x) u_n\|_{L^2}^2 \leq  \| u_n\|_{L^2( \{|x| \leq Ah^{1/2} \epsilon^{-1} (h) \}} \\
 \leq \| u_n\|_{L^2( \{|x| \leq  h^{1/2} \epsilon^{-2} (h) \}}  \Rightarrow \langle \mu, \chi (z)\rangle =0.
 \end{multline}

   To prove point~\ref{4.5}, recall {from Figure~\ref{fig.4}} that the damping $a$ is equal to $1$ on $(0, \frac 1 2) \times (0, \frac 1 2)$ and that 
  $$\| a u_n\|_{L^2}= \| a^{1/2} u_n\|_{L^2}  =o(1)_{n\rightarrow + \infty}.$$ Point~\ref{4.5} will follow from 
 { \begin{prop}
  Assume that 
  $$\| a u_n\|_{L^2}= o(1)_{n\rightarrow + \infty},$$
  and that the damping is equal to $1$ on $(0, \delta) \times (c,d) $ (resp. $(- \delta, 0) \times (c,d) $).
  Then the measure~$\mu$ vanishes two- microlocally on the right (resp on the left) above $T^*M\mid_{\{0\} \times (c,d)} $:
  \begin{equation}
  \begin{aligned}
  \mu \left (\left \{ (x,y,\xi, \eta, z,\zeta) ; x=0, y\in \left (c, d \right ), \eta =\pm 1,  z >0\right \} \right )&=0,\\
  (\text{resp.} \mu \left (\left \{ (x,y,\xi, \eta, z,\zeta) ; x=0, y\in \left (c, d \right ), \eta =\pm 1,  z <0\right \} \right )&=0),
  \end{aligned}
  \end{equation}
  \end{prop}}
 Let $\psi\in C^\infty( \mathbb{R})$ supported in $\{1<r\}$ and equal to $1$ for $r\geq 2$. Let $\chi \in C^\infty_0 (-1,1)$ equal to $1$ on $(-\frac 1 2, \frac 1 2)$, and $\widetilde{\chi} \in C^\infty_0 (c,d)$ equal to $1$ on $a+ \delta, b- \delta$, $\delta >0$.  Consider the symbol
 $$ b(x,y, \xi, \eta, z, \zeta)= \chi (\frac{2x} \delta)  \widetilde{\chi} (y) \chi(\xi) \chi ( \eta-1) \psi \left ( \frac{ z} { {\delta |}\zeta|}\right )\psi( z^2 + \zeta^2) 
 $$
 
 On the other hand, since $\chi (2x)  \widetilde{\chi}(y)$ is supported on $(-\frac 1 2, \frac 1 2) _x \times (c,d)_y $ and since $\psi ( \frac{ z} { {\delta} |\zeta|})$ is supported in $z>0$, {we infer that the range of $\text{Op}_{h_n} (b)$ is supported in the domain $(0, \frac 1 2) _x \times (c,d)_y $} and consequently
 \begin{multline}
 \bigl(\text{Op}_{h_n} (b) u_n, u_n\bigr) =\bigl(1_{x\in(0, \frac 1 2)} 1_{y \in (c,d)}  \text{Op}_{h_n} (b) u_n, u_n\bigr)
 = \bigl( \text{Op}_{h_n} (b) u_n, 1_{x\in(0, \frac 1 2)} 1_{y \in (c,d)} u_n\bigr)\\
 = \bigl(  \text{Op}_{h_n} (b) u_n, 1_{x\in(0, \frac 1 2)} 1_{y \in (0,\frac 1 2)}a u_n\bigr))= o(1)_{n \rightarrow +\infty}.\end{multline}
 This implies 
 $$ \mu \left (\left \{ (x,y,\xi, \eta, z,\zeta) ; x=0, y\in \left (c+ \delta,d - \delta \right ), \eta =1, z \geq 2{\delta |}\zeta|\right \}\right ) =0\ .$$
 Taking {$\delta >0$} arbitrarily small, we deduce that on the $(z,\zeta)$ sphere at infinity which contains the support of $\mu$, we have 
 $$ \mu \left (\left \{ (x,y,\xi, \eta, z,\zeta) ; x=0, y\in \left (c, d \right ), \eta =1,  z >0\right \} \right )=0.$$
 The case $\eta = -1$ and the {other properties} in~\eqref{apriori} follow similarly.
 
  To prove the last property, we write for $q \in S^0$
  \begin{multline}\label{commut1}
 \frac 1 {{2}ih_n}\Bigl [ h_n^2\Delta +1, \text{Op}_{h_n} (q)\Bigr] \\=   \text{Op}_{h_n} \bigl(( \xi \partial_x  + \eta \partial_y + \zeta \partial_z)q \bigr)- i 
 {\frac{h_n}{2}}  \text{Op}_{h_n} \bigl( \Delta_{x,y} a\bigr) 
-i{\frac{h_n}{2}} (\epsilon(h_n) h_n^{-1/2})^2 \text{Op}_{h_n} \bigl(    \partial_z ^2 q \bigr).
  \end{multline} 
 {   Since unfolding the bracket shows that, as $n\to \infty $, 
 $$\frac 1 {{2}ih_n}\left (\Bigl [ h_n^2\Delta +1, \text{Op}_{h_n} (q)\Bigr]u_n,u_n\right )\rightarrow 0 ,  $$
 we get }
  \begin{equation}\label{commut}
o(1)_{n\rightarrow \infty} = \Bigl( \text{Op}_{h_n} \bigl(( \xi \partial_x  + \eta \partial_y + \zeta \partial_z)q \bigr)u_n, u_n \Bigr) .
   \end{equation}
{Let us} compute the limit on the sphere at infinity of 
  $( \xi \partial_x  + \eta \partial_y + \zeta \partial_z)q$. 
  We denote by $\widetilde{q}$ the function $q$ in the $r, \theta$ coordinate system. In this system of coordinates, the operator $\zeta \partial _z$ reads 
  $$ - \sin^2 (\theta) \partial_\theta+ r \cos(\theta) \sin(\theta) \partial_r.$$
 {Now we use that, for a polyhomogeneous symbol  $q$  of {\em degree $0$},  the main part of $q$ at infinity does not depend on $r$. As a consequence,  the symbol $r \partial_r q$ is polyhomogeneous of {\em degree $-1$}} (while homogeneity would dictate {\em degree $0$}). Therefore we get, for any polyhomogeneous  symbol $q$ of degree $0$,  
  \begin{multline} \zeta \partial_z q \mid_{\mathbb{S}^{{1}}} = \lim_{r\rightarrow + \infty} 	\bigl(- \sin^2 (\theta) \partial_\theta+ r \cos(\theta) \sin(\theta) \partial_r \bigr)\widetilde{a} (x,y,\xi,\eta, r, \theta)\\
  = - \sin^2( \theta) \partial_\theta\lim_{r\rightarrow + \infty} \widetilde{q} (x,y,\xi,\eta, r, \theta).
  \end{multline}
   Since the measure $\widetilde{\mu}$ is supported in $\xi=0$,  equation~\eqref{eqprop} follows from~\eqref{commut}.
  
  \end{proof}
  
 We can now conclude the contradiction argument, and end the proof of the resolvent estimate~\eqref{obs}. Notice that the two fixed points for the flow of 
 $$\dot{\theta} = - \sin^2( \theta)$$
 are given by $\theta =0 (\pi)$. 
 We want to show that the measure $\widetilde{\mu}$ vanishes identically to get a contradiction with point \ref{4.1} in Proposition~\ref{prop-propag}. For $(x=0, \xi=0, y_0, \eta_0= \pm 1, \theta_0)$  in the support of $\widetilde{\mu}$, let us denote by $\phi_s( \theta_0)$ the solution of
 $$ \frac d {ds} \phi_s(\theta_0)= - \sin^2 ( \phi_s(\theta_0)), \quad \phi_0 ( \theta_0) = \theta_0,$$
 so that $\phi_s( \theta_0 ) = \text{Arccotan} ( s+ \text{cotan}( \theta_0))$. From the invariance~\eqref{eqprop} of the measure $\widetilde{\mu}$, we deduce that 
 $$ \forall s\in \mathbb{R},  (x=0, y_s=y_0 + s\eta_0\, (\text{mod}\, 2\pi) , \xi=0, \eta_0, \theta_s= \phi_s( \theta_0)) \in \text{supp}( \widetilde{\mu}).$$
 Consequently, if $\theta_0 \in [0, \pi)$, there exists $s>0 $ such that  $y_s \in (0, \frac 1 2)\, (\text{mod}\, 2\pi )$ while $\theta_s \in [0,\frac \pi 2)$, while, if $\theta_0 \in [- \pi, 0)$, there exists $s>0 $ such that $y_s \in (-\frac 1 2, 0)\, (\text{mod}\, 2\pi )$ while $\theta_s \in [- \pi, -\frac \pi 2)$. This is impossible according to~\eqref{apriori}.
 \ 
 \section{Back to the general case}\label{sec.3}

Let us work on the torus  {$ \mathbb{T}^2 =\mathbb R^2/ A\mathbb Z \times B\mathbb Z $ with $A>0, B>0$.} Since the irrational directions $\Xi= (A\xi , B\eta); \xi / \eta \notin \mathbb{Q}$ correspond to dense geodesics, and since $a$ is bounded from below on an open set, we deduce that the measure $\nu$ defined in Section~\ref{sec.2.1}  is supported --- in the $\Xi$ variables --- on the set of finitely many rational directions
{$$\Xi= (A\xi , B\eta)\ ;\  \xi / \eta \in \mathbb{Q}\ ,$$ 

satisfying moreover the elliptic regularity condition,
$\vert \Xi \vert^2=1$}, which do not enter the interior of the rectangles. Hence,  there exists an isolated direction $\Xi_0$, so that $(X_0, \Xi_0) \in \text{ supp } (\nu)$, which can be written as
 \begin{equation}
\label{eq:Xi0}   \Xi_0 = \frac1 { \sqrt { n^2 A^2 + m^2 B ^2 } } ( n A
, m B  )\, ,
\ \ \Xi_0^\perp  = \frac1 { \sqrt { n^2 A^2 + m^2 B^2 } } ( - m B ,
n    A  ) \, ,\end{equation}
{where the integers $n,m$ may be chosen to have gcd $1$.}
The change of coordinates in  $\mathbb{R}^2 $, 
\begin{equation}
\label{eq:F}  F \; : \; ( x, y ) \longmapsto  X = F ( x, y )  = x
\Xi_0^\perp  + y \Xi_0 \,, 
\end{equation}
is orthogonal and hence $ - \Delta_X = D_x^2 + D_y^2 $. 

We have the following simple lemma (see~\cite[Lemma 2.7]{BuZw12}), {which can be deduced from an elementary calculation}.
\begin{lemm}\label{lem.2.1}
\label{l:si}
Suppose that $ \Xi_0 $ and $ F $ are given by \eqref{eq:Xi0} and
\eqref{eq:F}. If 
 $ u = u ( {x,y} ) $ is periodic with respect to $ A \mathbb{Z} \times B  \mathbb{Z}$
then {$F^*u:=u\circ F$ satisfies}
\begin{equation}
\label{eq:per}
 F^* u ( x + k a , y + \ell b ) = F^* u ( x , y - k \gamma ) \,, \ \ k,
 \ell \in \mathbb{Z} \,,  \ \ ( x, y ) \in \mathbb{R}^2 \,, 
\end{equation} 
where, for  fixed $ p , q \in \mathbb{Z} $ such that $q n - p m=1$, 
\[ a =   \frac { A B } { \sqrt {  n^2 A^2 + m^2 B ^2 } }    \,,
\ \      b =  \sqrt {  n^2 A^2  + m^2 B^2 }  \,, \ \  \gamma = -
\frac{ pn A^2 + q m B^2  }{  \sqrt { n^2 A^2 + m^2 B^2 } }  \,. \]
When $ B/A = r/s \in \mathbb{Q} $, $r,s \in \mathbb{Z} \setminus \{0\}$, then 
\begin{equation}\label{chgvar}
 F^*{u} ( x + k \widetilde a , y + \ell b ) = F^* u ( x, y ) \,, \ \  \ \ k,
 \ell \in \mathbb{Z} \,,  \ \ ( x, y ) \in \mathbb{R}^2 \,, 
 \end{equation}
for $ \widetilde a = ( n^2 s^2 + m^2 r^2 )  a  $.

\end{lemm}
In this new coordinate {system}, we know that there exists $x_0$ such that $(x_0, y_0, 0, 1)$ is in the support of the measure $F^*\mu$. By translation invariance, we can assume that $x_0=0$. Since $(\xi \partial_x+\eta \partial_y) F^*\mu =0$, we infer that actually the whole line $(x_0=0, \mathbb{R}(\text{mod} 2\pi), 0, 1)$ belongs to the support of $F^*\mu$.  If this bicharacteristic curve enters the interior of the support {of $a$}(i.e. encounters a point in a neighborhood of which $a$ is bounded away from $0$), then by propagation, no point of this bicharacteristic curve lies in the support of $\mu$ which gives a contradiction. On the other hand since assumption~\ref{geom} is satisfied, we know that there exists two (at least) {polygons} $R_1, R_2$ so that the right side of $R_1$ is $\{0\} \times [\alpha, \beta] $ while the left side of $R_2$ is $\{0\} \times [\gamma, \delta]$. {We may shrink these polygons to rectangles having the same property}.
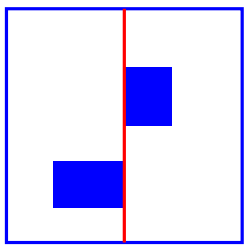
\begin{figure}[ht]
 \begin{center}
\begin{tikzpicture}[scale=1.55]
\fill[blue] (-.6,-.7) -- (-.6,-.3) -- (0,-.3) -- (0,-.7) -- cycle;
\fill[blue] (.4,0) -- (.4,.5) -- (0,.5) -- (0,0) -- cycle;
\draw[very thick, blue] (-1,-1) -- (-1,1) -- (1,1)--(1,-1)--cycle;
\draw[very thick, red] (0,-1) -- (0,1);
\end{tikzpicture}
\end{center}
\caption{The microlocal model: on the left the rectangle $R_1$, on the right the rectangle $R_2$, in the middle the bicharacteristic in the support of $\mu$}
\end{figure}

 In other words, we are microlocally reduced to the study of the  checkerboard in Figure~\ref{fig.4}. Notice that the change of variables we used in Lemma~\ref{lem.2.1} does not keep periodicity with respect to the $x$ variables but transforms it into some pseudo-periodicity condition (see~\eqref{eq:per}). However, for the study of the checkerboard model in Section~\ref{sec.2}, we  only used periodicity with respect to the $y$ variables (to prove Proposition~\ref{aprioriprop})--- which is preserved. The rest of the contradiction argument follows the same lines as in Section~\ref{sec.2}.

 \section{Generalized geometric condition} \label{app.A}
For a general {Riemannian} manifold and a general damping function $a \in L^\infty(M)$,  a natural substitute to~\eqref{gcc} is the following generalized geometric condition.
\begin{equation*}\tag{GGCC}\label{ggcc}
\exists T, c>0\ :\  \liminf_{\epsilon\rightarrow 0} \inf_{\rho_0 \in S^*M} \frac{ 1} {\text{Vol}( \Gamma_{\rho_0, \epsilon,T})}\int_{\Gamma_{\rho_0, \epsilon,T}}  a(x)dx \geq c,
\end{equation*}
where $\Gamma_{\rho_0, \epsilon,T}$ is the set of points at distance less than $\epsilon$ from  the geodesic segment $\{\gamma_{\rho_0} (s), s\in (0,T)\}$. 
At first glance,~\eqref{ggcc} might seem to be a strong condition, difficult to fulfill. We shall prove below that it cannot be relaxed as, on any manifold and for any $a \in L^\infty (M)$, it is a necessary condition for uniform stabilisation. On the other hand, we also prove below that in the case of two dimensional tori it is equivalent to Assumption~\ref{geom}.  
We conjecture that on a general manifold and for general $a\in L^\infty$, uniform stabilisation holds if and only if~\eqref{ggcc} holds. {The results in this article show that it is indeed the case on two dimensional tori,  if $a$ satisfies~\eqref{somme}. For general, dampings it is easy to show that \eqref{ggcc} implies~\eqref{wgcc}, while the compactness of $S^*M$ shows that it is implied by ~\eqref{sgcc} ($\delta$ in ~\eqref{sgcc} can,  by compactness, be chosen the same for al $\rho_0 \in S^*M$). }
\subsection{The generalized geometric condition is necessary for stabilisation}
\begin{theo}\label{theorem.4} Uniform stabilisation implies~\eqref{ggcc}.
\end{theo}
\begin{proof}
The proof of this results relies on geometric optics constructions (with complex phases) for the wave equation by Ralston~\cite[Section 2.1]{Ra82} that we recast in our wave equation context.
\begin{prop}\label{ralston}
Let  $M$ be a compact manifold without boundary endowed with a smooth metric $g$ and a smooth density $\kappa$. Let 
\begin{equation}\label{delta}
\Delta = \text{div}_{\kappa} \nabla_g 
\end{equation}
be the Laplace operator.  Let $(t_0, x_0,\tau_0= \frac 1 2, \xi_0) \in \text{Char} (\partial_t^2 - \Delta)$, the characteristic manifold. Denote by $(t(s)= t_0 +s, \tau_0= \frac 1 2 , \gamma(s), \xi(s)) $ the bicharacteristic starting from $(t_0, x_0,\frac 1 2, \xi_0)$. 
Then for any $N>0$ there exists a family of approximate solutions $v_{h,N}(t,x)$ defined for $0< h<h_0$ to the wave equation 
\begin{equation}\label{approximate}
(\partial_t^2 - \Delta) v_{h,N}= \mathcal{O} ( h^N)_{L^2(M)}, E(v_{h,N}) = \int_M ( |\nabla_g v_{h,N}|^2 + |\partial_t v_{h,N}|^2 ) \kappa dx =1 + o(h)
\end{equation} 
with error terms locally uniformly controlled in time, and which are (locally in time) exponentially localised in $\mathbb{R}_t \times M$ near $(t(s), x(s)$: 
\begin{equation}  \label{upper} 
\begin{gathered}
\forall T> 0, \exists C, \alpha >0; \forall t\in [0,T], \forall x \in M,\\
 \bigl( |v_{h,N} | + |h \nabla_x v_{h,N}| + |h \partial_t v_{h,N}|\bigr) (t(s),x)  \leq C h^{1 - \frac{ d}{ 4}}e^{- \alpha \frac{ \text{dist} ( x, x(s) )^2} {h}}
 \end{gathered}
\end{equation}
and consequently, if we denote by $\Gamma_T=\gamma([0,T]) $ the image of the geodesic in $M$,
\begin{equation}  \label{upperbis} 
\begin{gathered}
\forall T> 0, \exists C, \alpha >0;\forall x \in M,\\
 \int_0^T \bigl(|\nabla_x v_{h,N}|^2 + |\partial_t v_{h,N}|^2\bigr) (t,x) dt \leq C h^{- \frac{ d-1}{ 2}}e^{- \alpha \frac{ \text{dist} ( x, \Gamma_T) ^2} {h}}.
 \end{gathered}
\end{equation}
\end{prop}
Let us first show how we can deduce Theorem~\ref{theorem.4} from Proposition~\ref{ralston}. We are going to test the observation estimates~\eqref{eq.nonamortie} on such sequences of solutions.

{Let us we assume that~\eqref{ggcc} does not hold. Fix $T>0$. Then there exists $\eta_n = (x_n, \xi_n) \in S^* \mathbb{T}^d$, $\epsilon _n \rightarrow 0$ such that, with 
$$\lim_{n\rightarrow +\infty} \kappa_n =0\ ,\  \kappa_n :=  \frac{1} { \epsilon _n ^{ d-1}} \int_{\Gamma_{\eta_n, \epsilon_n ,T}}  a(x)dx .$$
Let $t_n =0$.
 Let $\rho_n = (t_n, \tau_n = \frac 1 2, x_n,\xi_n)$, fix $N=1$ (we actually need a crude version of Ralston construction) and $v_{h}^n$ be the approximate solution of the wave equation constructed in Proposition~\ref{ralston}, with initial point $\rho_n $. We shall use that the family of solutions which depends on two parameters $h$ and the initial point in the cotangent bundle is uniformly controled with respect to this latter parameter, which will follow from the proof of Proposition~\ref{ralston} given below. Since, according to Proposition~\ref{ralston}, 
we have
$$\| (v_{h}^n, \partial_t v_{h}^n)\mid_{t=0}\|_{H^1\times L^2}= 1+ o(1)_{n\rightarrow + \infty},$$ and according to~\eqref{approximate} and Duhamel formula, $w_h^n$ is, modulo a $\mathcal{O}(h)$ error in energy space, equal to the solution to the exact wave equation with the same initial data,  to show that uniform stabilisation does not hold, it is now enough to show that for a properly chosen sequence  $h_n \rightarrow 0$
 \begin{equation}
 \lim_{n\rightarrow + \infty} \int_{0}^T \int_M a(x)  |\partial_t v_{h_n,n}|^2 dx dt =0
 \end{equation}
  Extracting a subsequence, we can assume that the sequence of initial points $\rho_n$ converges to $\rho= ( t_0=0, \xi_0= \frac 1 2, x_0,\xi_0)$. 
 The only point we shall use about our approximate solutions is the upper bound~\eqref{upperbis}, which implies
 }
   
\begin{multline}\label{eq.estim}
\int_{0}^T \int_M a(x)  |\partial_t v_{h_n}|^2 dx dt = \int_{0}^T \int_{M} a(x)  |\partial_t v_{h_n}|^2 dx dt \\
\leq C  \int_{\Gamma_{\rho_n, \epsilon_n ,T}} a(x)  h_n^{- \frac{ d-1} 2} e^{- \alpha \frac{ \text{ dist } (x, \Gamma_{\rho_n, T}) ^2 } {h_n}} dx+\int_{\Gamma^c_{\rho_n, \epsilon_n ,T}} a(x) 
h_n^{- \frac{ d-1} 2} e^{- \alpha \frac{ \text{ dist } (x, \Gamma_{\rho_n, T}) ^2 } {h_n}} dx .
\end{multline}
The contribution of the first term is bounded by 
$$ C h_n ^{- \frac {d-1} 2}\int_{\Gamma_{\rho_n, \epsilon_n ,T}} a(x)dx \leq \kappa_n \bigl( \frac{ \epsilon _n^2 } { h_n }\Bigr) ^{\frac {d-1} 2}.
$$ 

On the other hand, the second term is bounded by 
\begin{equation}\label{courbe}
 \| a \|_{L^\infty} \int_{\Gamma^c_{\rho_n, \epsilon_n ,T}} 
h_n^{- \frac{ d-1} 2} e^{- \alpha \frac{ \text{ dist } (x, \Gamma_{\rho_n, T}) ^2 } {h_n}} dx .
\end{equation}
To estimate this integral we work in (a finite set of) coordinate systems. In such local coordinates, $\Gamma_{\rho_n, T}$ is a finite union of smooth arcs of geodesics (because the geodesic can self intersect) and it is enough to estimate~\eqref{courbe} where we replaced $ \text{ dist } (x, \Gamma_{\rho_n, T})$ by the distance to any such arc. We can change again coordinates such that locally the considered arc of geodesic is 
$$ \{ (y_1=0, y'\in \mathbb{R}^{d-1})\},$$
and the distance to the arc $\gamma_n$ satisfies 
$$ \exists C>0; \frac 1 C |y'| \leq \text{dist} (y, \gamma_n)\leq C |y'|$$
This leads to the estimate (if $\epsilon _n \geq C^2 \sqrt{h_n}$)
\begin{multline}
\int_{\text{dist} (x, \gamma) \geq \epsilon_n} 
h_n^{- \frac{ d-1} 2} e^{- \alpha \frac{ \text{ dist } (x, \gamma_n) ^2 } {h_n}} dx 
= \int_{|x'| \geq \frac{\epsilon_n } C} h_n^{- \frac{ d-1} 2} e^{- \alpha \frac{|x'|^2 } {C h_n}} dx \\
 =C'\int_{|y'| \geq \frac{\epsilon_n } {C^2 \sqrt{h_n}}} e^{-  \alpha |y'|^2 } dy' \leq C' e^{-\alpha \frac{\epsilon_n ^2} {C^4 h_n}}
\end{multline}

We now choose $h_n= \kappa_n ^{\frac 1 {d-1}} \epsilon _n^2  \rightarrow 0$ such that 
\begin{equation}
\frac{ \epsilon_n ^2} {h_n } = \kappa_n ^{-\frac 1 {d-1}} \rightarrow + \infty, \qquad \kappa_n \bigl(\frac { \epsilon_n ^2} {h_n }\bigr)^{ \frac {d-1} 2}= \kappa_n ^{\frac 1 2}  \rightarrow 0.
\end{equation}
This choice implies 
$$\int_{0}^T \int_M a(x)  |\partial_t w_{h_n}|^2 dx dt = o(1)_{n\rightarrow + \infty},$$
which contradicts~\eqref{eq.nonamortie} because the energy of the initial data $(w_{h_n}, \partial _t w_{h_n})$ is constant and nonzero. This completes the proof of Theorem~\ref{theorem.4}.\end{proof}

Let us now come back to the proof of Proposition~\ref{ralston}. This is basically done in~\cite[Section 2.1]{Ra82}. The idea is to define oscillating solutions (phase and symbol) by constructing the germs on the bicharacteristic curve.  
Let $\rho_0= (t_0, \tau_0 = \frac 1 2, x_0, \xi_0)$ a point in the characteristic variety of the wave equation
$$ \text{Char}= \{ (t, x,\tau,\xi) \in T^*(\mathbb T^d) :  |\tau|^2 = |\xi|^2=1 \}.$$
Let $\Gamma= \{ t(s), \gamma(s), \tau (s) = \frac 1 2, \xi (s)) $ be the bicharacteristic curve issued from $\rho_0$.
 For any $T<+\infty$, we can choose  systems along the geodesic $\gamma$  and get an immersion 
$$ i:  ( - \epsilon, T+ \epsilon)\times B(0, \epsilon) \subset  \mathbb{R}\times \mathbb{R}^{d-1}\rightarrow M, $$
along which the bicharacteristic takes the form 
$$\gamma(s) = (t=s,x_1=s,x'= 0, \tau = \frac 1 2, \xi_1 = - \frac 1 2, \xi' =0),$$
 which allows to reduce the analysis to $\mathbb{R}^d$.  In this coordinate system,~\eqref{delta} takes the form 
 $$ \Delta = \frac 1 {\kappa(x)} \sum_{i,j} \frac{ \partial } {\partial x_i} g^{i,j} (x) \kappa (x) \frac{ \partial}{ \partial x_j}.$$

We now write $y = (t,x)$ and seek approximate solutions of the wave equation with the form 
\begin{equation}\label{quasi-mode}
 u_h(t,x) = e^{\frac i h \psi(t,x)} \sigma(t,x,h), 
 \end{equation}
where $\sigma(t=0)=\sigma_0\in C^\infty_0 ( \mathbb{R}^d)$ has sufficiently small compact support near $0$. 
Applying the operator $\partial_t^2 - \Delta_x$ we get 
\begin{equation}\label{eq.gaussien}
\begin{aligned}
(\partial_t^2 - \Delta_x) u_h = & - \frac{ 1} {h^2} \Bigl ( (\partial_t \psi )^2 - \sum_{1\leq k,j\leq n}g^{k,j} (x) \partial_{x_k} \Psi \partial_{x_j} \Psi\Bigr) \sigma e^{\frac i h \psi}\\ 
& + e^{\frac i h \psi}\frac i h \Bigl ( 2 \partial_t \psi \partial_t \sigma   - 2 \sum_{1\leq k,j\leq n} g^{k,j} (x) \partial_k \psi \partial_j \sigma - \frac 1 {\kappa} \sum_{1\leq k,j\leq n} \partial_k ( g^{k,j} \kappa) (x)  \sigma \partial_j \psi     \Bigr)  \\
&+ e^{\frac i h \psi}\Bigl( \partial^2_t -  \Delta \psi\Bigr)\sigma 
\end{aligned}
\end{equation}
In~\cite[Section 2.1]{Ra82}, Ralston then shows that provided that 
$$ \psi (t(s), x(s) ) = t(s) - x_1 (s)+ cste\Leftrightarrow \partial_{t,x} \psi ( t(s), x(s) ) = ( \tau(s), \xi(s)), $$ and choosing 
\begin{equation}
 \Im \Bigl( \frac{ \partial^2 \psi} { \partial{x^2}} \Bigr)\mid_{\gamma(0)} \geq c\text{Id}, c>0
 \end{equation}
  it is possible to solve both the e\"ikonal equation 
$$p = \Bigl ( (\partial_t \psi )^2 - \sum_{1\leq k,j\leq n}g^{k,j} (x) \partial_{x_k} \Psi \partial_{x_j} \Psi\Bigr) =0$$
and the transport equation
$$T= \bigl ( 2 \partial_t \psi \partial_t \sigma   - 2 \sum_{1\leq k,j\leq n} g^{k,j} (x) \partial_k \psi \partial_j \sigma - \frac 1 {\kappa} \sum_{1\leq k,j\leq n} \partial_k ( g^{k,j} \kappa) (x)  \sigma \partial_j \psi     \bigr) )e^{\frac i h \psi}  +h \bigl( \partial^2_t -  \Delta \psi\bigr)\sigma =0, $$
with 
\begin{equation}\label{def-pos}
 \Im \Bigl( \frac{ \partial^2 \psi} { \partial{x^2}} \Bigr)\mid_{\gamma(s)} \geq c(s)\text{Id}, c(s)>0
 \end{equation}
to arbitrary large  order on the bicharacteristic $\gamma$ by choosing 
$$ \sigma = \sum_p h^p \sigma _p.$$ Here by solving to arbitrary large order, we mean that we can cancel an arbitrary large number of $(t,x)$ derivatives on $\gamma$.

On the torus $\mathbb{T}^n$, these constructions can be performed explicitely and we get 
\begin{equation}
\label{phase}
\psi(t,x)=   t- x_1  +i \bigl (( t- x_1)^2 + g(t) |x'|^2 \bigr) + O( |t- x_1|^3 + |x'|^3), 
\end{equation}
with $g$ solving 
$$ 2i g'(t) + 4 g^2 (t) =0 \Leftarrow g(t) = \frac{ g(0) } { 1- 2i t g(0)}\ ,\ {g(0):=1}\ .$$
Notice in particular that 
\begin{equation}\label{partie-reelle}
\Re (g(t) ) = \frac{ 1} {1+ 4t^2}>0
\end{equation}
and we can choose a symbol 
$$
 \sigma (t, x_1, x') =  \frac{\sigma_0 (t-x_1, x')}{ (1- 2it) ^{{\frac{n-1}2}}} + O(h) + O(|t-x_1| + |x'|). $$
Finally, it remains to cut off the symbol such constructed near the geodesic (taking benefit from~\eqref{def-pos}, we see that this troncature will add an exponentially small error), and to normalise by multiplying by 
$$c h^{1 - \frac d 4}$$ to ensure the normalisation of the energy in~\eqref{approximate} and the error bound~\eqref{upper}. We leave the details to the reader.

\subsection{Assumption~\texorpdfstring{\ref{geom} and~\eqref{ggcc}}{(GEOM) and (GGCC}}
On $2$-d tori and for dampings $a$ satisfying~\eqref{somme} we have
\begin{prop}\label{equiv}
 On a two dimensional torus $\mathbb{T}^2$, if the damping $a$ satisfies~\eqref{somme}, then \eqref{ggcc} is equivalent to Assumption~\ref{geom}.
 \end{prop}
 \begin{proof}
 Since Assumption~\ref{geom} implies uniform stabilisation (Theorem~\ref{theorem.1}) which in turn implies~\eqref{ggcc} (Theorem~\ref{theorem.4}), is is enough to show  that~\eqref{ggcc} implies Assumption~\ref{geom}. 
 
Let us assume~\eqref{ggcc}. If Assumption~\ref{geom} was not satisfied, then there would, for any $T>0$ exist a geodesic curve $\gamma$ of length $T$ which either does not encounter $ \mathcal{R}= \cup_{j=1} ^N \overline{R_j}$, or does encounter $\mathcal{R}$ only at corners, or encounter  $\mathcal{R}$, only on the left (or only on the right). In the first case, then by compactness, the geodesic curve remains at distance $\epsilon _0$ of $\mathcal{R}$, and consequently for $0< \epsilon <\epsilon _0$, then 
   $$\int_{\Gamma_{\rho_0, \epsilon,T}}  a(x)dx=0.$$
   In the second case (see  checkerboard in Figure~\ref{fig.1}.b), by compactness, the geodesic curve encounters only a finite number of corners, and consequently ($d=2)$
   $$\int_{\Gamma_{\rho_\sigma, \epsilon,T}}  a(x)dx=\mathcal{O} ( \epsilon ^2),$$
   with a constant $c>0$ depending on the angles of the corners, while 
   $$ {\text{Vol}( \Gamma_{\rho_0, \epsilon,T})} \sim C \epsilon 
   $$
   which implies that~\eqref{ggcc} does not hold. 
   In the last case (see the right checkerboard in Figure~\ref{fig.1}.c), let us consider the family of geodesics $ \gamma_\sigma  =\{\gamma_{\rho_\sigma} (s), s\in (0,T)\} , \sigma \in [0,1),$ parallel on the right to $ \gamma_0 =\{\gamma_{\rho_0} (s), s\in (0,T)\} $ (i.e if $\rho_0 = (X_0, \Xi_0)$, then $\rho_\sigma= X_0+ \sigma \Xi_0^\perp$, where $\Xi_0^\perp$ is the unit vector orthogonal to $\Xi_0$, pointing on the right of $\gamma_0$). 
 Since on the right $\gamma_0$ encounters no side of any rectangle $R_j$, it may encounter only (finitely many) corner points. As a consequence, for any $\sigma>0$ sufficiently small, and $0< \epsilon \ll \sigma$,  
 $$\int_{\Gamma_{\rho_\sigma, \epsilon,T}}  a(x)dx \sim c \sigma \epsilon \qquad ( \epsilon \rightarrow 0),$$
 $$\text{Vol}( \Gamma_{\rho_\sigma, \epsilon,T}) \sim C \epsilon, \qquad ( \epsilon \rightarrow 0).$$
 We deduce that 
 $$\lim_{\epsilon \rightarrow 0} \frac{ 1} {\text{Vol}( \Gamma_{\rho_\sigma, \epsilon,T})}\int_{\Gamma_{\rho_\sigma, \epsilon,T}}  a(x)dx = c \sigma,$$
 letting $\sigma \rightarrow 0$  shows that~\eqref{ggcc} does not hold.  
 \begin{figure}[h]
 \begin{center}
 \begin{tikzpicture}[scale=1.4]
\fill[blue] (-1,1) -- (-.5,1) -- (-.5,-1) -- (-1,-1) --  cycle;
\fill[blue] (1,1) -- (1,-1) -- (0, -1)--(-.5,0)--(0,1) -- cycle;
\draw[very thick, blue] (-1,-1) -- (-1,1) -- (1,1)--(1,-1)--cycle;
\draw[thick, red] (-0.5,-1) -- (-0.5,1);
\draw[very thick, red, dashed] (-0.4,-1) -- (-0.4,1);
\draw (-0.55, 1.1) node[above] {$\gamma$};
\draw (-0.35, 1.1) node[above] {$\gamma_\sigma$};
\draw (-0.35, 1.1) node[above] {$\gamma_\sigma$};
\draw[thick, ->] (-1.4,0) -- (1.4,0);
\draw[thick, ->] (0,-1.1) -- (0,1.4);
\draw (1.1,0) node[below] {$1$};
\draw (1.4,0) node[below] {$x$};
\draw (-.58,-1.1) node[below] {$-\frac 1 2$};
\draw (0, -1.15) node[below] {$0$};
\draw (0,1.4) node[left] {$y$};

\end{tikzpicture}
\end{center}
\caption{}
\end{figure}
  \end{proof}
  \appendix
  
     \section{Resolvent estimates and stabilisation}
  In this appendix, we collect a few classical results on resolvent estimates.
  \subsection{Resolvent estimates and stabilisation}  It is classical~\cite{Ge78} that stabilisation or observability of a self adjoint evolution system is equivalent to resolvent estimates (see also~\cite{BuZw03, Mi12,AL14}). For completness we shall give below a proof (only the fact that resolvent estimates imply stabilisation). 
 \begin{prop}\label{prop.resolvent}
 Consider a strongly continuous semi-group $e^{tA}$  on a Hilbert space $H$, with  infinitesimal generator $A$ defined on $D(A)$. 
  The following two properties are equivalent
  \begin{enumerate}
  \item \label{un}There exists $C, \delta>0$ such that the resolvent of $A$, $(A-\lambda)^{-1}$ exists for $\Re \lambda \geq  - \delta $\ and satisfies 
 $$ \exists C>0; \forall \lambda \in \mathbb{C}^\delta = \{ z \in \mathbb{C}; \Re z {\ge } - \delta\}, \| (A- \lambda)^{-1}\|_{\mathcal{L} ( H)}  \leq C.$$
\item \label{deux}There exists $M, \delta>0$ such that for any $t>0$
 $$ \| e^{tA} \|_{\mathcal{L} (H)} \leq M e^{- \delta t}.$$
 \end{enumerate}
 \end{prop}
 \begin{proof}
 Let us first prove that \eqref{deux} implies~\eqref{un}. We start with the following resolvent equality (always true for for $\Re \lambda \ll 0$),
 $$ (A- \lambda)^{-1}f = - \int_0 ^{+\infty} e^{t(A- \lambda) } f dt,
 $$
 and we deduce that if $\| e^{tA} \| \leq C e^{ -\beta t}$, \eqref{un} is satisfied  for any $\delta < \beta$. 
 To prove that~\eqref{un} implies~\eqref{deux}, for $u_0 \in D(A)$, and $\chi \in C^\infty( \mathbb{R})$ {equal to $0$ for $t\leq -1$ and to $1$ for $t \geq 0$}, consider $$u(t)= \chi (t) e^{t(A-\omega) }u_0.$$ For $\omega$ large enough, $u$ belongs to $L^\infty(\mathbb{R}; H)$, because strongly continuous semi-groups of operators satisfy 
 $$ \exists C,c >0; \forall t>0, \| e^{tA}\| \leq C e^{ct}, $$ and $u$ satisfies 
 $$ (\partial_t + \omega -A) u (t)= \chi'(t) e^{t(A-\omega )} u_0 =: v(t).$$
 Taking Fourier transforms in the time variable, we get 
 \begin{equation}\label{fourier}
  ( i \tau + \omega - A) \widehat{u}(\tau ) = \widehat{v}(\tau).
  \end{equation}
 Since $v(t)$ is supported in $t\in [-1,0]$, the r.h.s. in~\eqref{fourier} is holomorphic and bounded in any domain 
 $$\mathbb{C}_\alpha= \{ \tau \in \mathbb{C} : \Im \tau \geq \alpha, \alpha \in \mathbb{R}\}.$$
 From the assumption on the resolvent, we deduce that $\widehat{u}$ admits an holomorphic extension to $\{ \tau: \Im \tau {\leq } \delta + \omega\}$ which satisfies 
 $$ \| \widehat{u} ( \tau ) \|_{H} \leq C \|\widehat{v}(\tau) \|_{H} . $$
 We deduce that 
 \begin{multline}
  \| e^{(\omega + \delta)t } u \|_{L^2( \mathbb{R}_t; H))} = \| \widehat{u} ( \tau + i (\omega+ \delta))\|_{L^2( \mathbb{R}_\tau; H) } \leq C \|\widehat{v}( \tau + i (\omega+ \delta))\|_{L^2( \mathbb{R}_\tau; H) }\\
  \leq C \| e^{(\omega + \delta)t }v\|_{L^2( \mathbb{R}_t; H) }\leq C' \| u_0 \|_H.
  \end{multline}
 This implies exponential decay { of $e^{tA}u_0$} in the $L^2_t$ norm, { with the weight $e^{\delta t}$}. Now consider $w(t) :=\chi(t-T) e^{tA}u_0$, which satisfies 
 $$  (\partial_t -A) w = \chi'(t-T) e^{tA} u_0, w\mid_{t=T-1}=0.$$
 
From Duhamel formula, we deduce
$$ w(T) = \int_{T-1}^T e^{(T-s)A}  \chi'(t-T) e^{sA} u_0 ds,$$
and consequently (recall that the semigroup norm is locally bounded in time)
\begin{multline}
\| w(T)\|_{H}\leq \int_{T-1}^T \| e^{(T-s)A}\chi'(t-T) e^{sA} u_0\|_H ds\\\
\leq C \sup_{\sigma\in [0,1]} \| e^{\sigma A}\|_{\mathcal{L}(H)} \int_{T-1}^T \| e^{sA} u_0\|_H 
 \leq C' e^{- \delta T} \| e^{\delta s } e^{sA} u_0\|_{L^2(T-1,T); H}\\ \leq C'' e^{- \delta T} \| u_0\|_{H}.
\end{multline}

 \end{proof}

 \subsection{Semi-groups for damped wave equations}\label{sec.B.2.1}
 The solution to~\eqref{eq.amortie} is given very classically by 
 $$ \begin{pmatrix} u \\ \partial_t u \end{pmatrix} = e^{tA} \begin{pmatrix} u_0 \\ u_1 \end{pmatrix}, \qquad A= \begin{pmatrix} 0 & \text{Id} \\ \Delta - m & -a \end{pmatrix}$$
 where $A$
 is defined on $\mathcal{H}=H^1(M) \times L^2(M)$ with domain $H^2(M) \times H^1(M)$. When $m>0$, since 
 $$ E(u) = \| u\|_{H^1} ^2 + \| \partial_t u \|_{L^2}^2,$$ to study the decay of the energy, we can apply directly the caracterization given by Proposition~\ref{prop.resolvent}. When $m=0$, the semi-group $e^{tA}$ is no more a contraction semi-group on $H^1\times L^2$ (because the energy~\eqref{energy} does not control the $H^1$ norm). The main difference from the case $m=0$ and $m$ non trivial comes from 
 \begin{lemm}
 Assume that $0 \leq m\in L^\infty(M)$ and $m$ is not trivial ($\int_M m(x) dx >0$). Then the norms
 $$ \| u \|_{H^1} = \Bigl(\| \nabla_x u \|_{L^2}^2+ \| u \|_{L^2}^2 \bigr)^{1/2}, \qquad \| u\| = \sqrt{ E(u)} = \Bigl ( \| \nabla_x u \|_{L^2}^2+ \| m^{1/2} u \|_{L^2}^2 \bigr)^{1/2} $$
 are equivalent
 \end{lemm}
 Indeed, as for a classical  proof of Poincar\'e inequality, we proceed by contradiction to prove the only non trivial inequality ($\| u \|_{H^1} \leq C \| u\|$), and get a sequence 
 $(u_n) \in H^1(M)$ such that 
 $$  \| u_n\|_{H^1} =1, \qquad \| u_n \|\rightarrow_{n\rightarrow + \infty} 0$$
  By the weak compactness of the unit ball in $H^1$ we can extract a subsequence (still denoted by $(u_n)$ which converges weakly in $H^1$ (and hence because $M$ is compact strongly in $L^2$ to a limit $u \in H^1$. Since $\| u_n \|\rightarrow + \infty$ we get that the sequence actually converges strongly in $H^1$ and 
  $$ \| u\| =0 \rightarrow \nabla_x u =0, m^{1/2} u =0.$$
  We deduce that $u$ is constant in $M$ and since $\int_M m u =0$, we finally get $u =0$ which contradicts the fact that $\| u_n \|_{H^1} =1$ (and the strong convergence of $u_n$ to $0$).
  
 For $s= 1,2$,  $\dot{H}^s = H^s(M) / \mathbb{R}$ the quotient space of $H^s(M) $ by the constant functions, endowed with the norm 
 $$ \| \dot{u} \| _{\dot{H}^1} = \| \nabla u\|_{L^2}, \qquad \| \dot{u} \| _{\dot{H}^2} = \| \Delta u\|_{L^2}.$$
 We define the operator 
  $$ \dot{A}= \begin{pmatrix} 0 & {\Pi} \\ \dot{\Delta} & -a \end{pmatrix}$$ 
  on $\dot{H}^1 \times L^2$ with domain $\dot{H}^2 \times H^1$, where $\Pi$ is the canonical projection $H^1 \rightarrow \dot{H}^1$ and $\dot{\Delta}$ is defined by 
  $$ \dot{\Delta} \dot{u} = { \Delta u}$$ (independent of the choice of $u\in \dot{u}$). The operator $\dot{A}$ is maximal dissipative and hence defines a semi-group of contractions on $\dot{\mathcal{H}} = \dot{H^1} \times L^2$. Indeed for $U= \begin{pmatrix}\dot{u}\\ v\end{pmatrix}$, 
  $$ \Re \bigl( \dot{A} U, U\bigr)_{\dot{\mathcal{H}}} = \Re ( \nabla u, \nabla v)_{L^2} + ( \Delta u -a v, n)_{L^2} = - ( av,v) _{L^2},$$
  and 
  \begin{equation}
\begin{aligned} (\dot{A} -\text{Id} ) \begin{pmatrix}\dot{u}\\ v\end{pmatrix} = \begin{pmatrix}\dot f\\ g\end{pmatrix} &\Leftrightarrow \Pi v - \dot u = \dot f, \dot{\Delta} \dot{u} -(a+1) v = g\\
 & \Leftrightarrow \Pi v - \dot u = \dot f, \Delta v - (1+a) v = g + \Delta f \in H^{-1} (M)
  \end{aligned}
  \end{equation}
  and we an solve this equation by variational theory. Notice that this shows that the resolvent $(\dot{A}- \text{Id} )^{-1}$ is well defined and continuous from $\dot{H}^1 \times L^2$ to $\dot{H}^2 \times H^1$. 
  \begin{lemm}
  The injection $\dot{H}^2 \times {H}^1$ to $\dot{H}^1 \times L^2$ is compact
  \end{lemm}
  This follows from identifying $\dot{H}^n$ with the kernel of the linear form $u \mapsto \int_{M} u$). 
  
  \begin{coro}\label{coroA.4}
  The operator $(\dot{A}- \text{Id} )^{-1} $ is compact on $ \dot{\mathcal{H}}$
  \end{coro}
  On the other hand, it is very easy to show that for $(u_0, u_1) \in H^1\times L^2$, 
  $$\begin{pmatrix} \Pi & 0 \\ 0 & \text{Id} \end{pmatrix} e^{tA} = e^{t\dot{A}} \begin{pmatrix} \Pi & 0 \\ 0 & \text{Id} \end{pmatrix}, $$ and consequently, stabilisation is equivalent to the exponential decay (in norm) of $e^{t\dot{A}}$ (and consequently, according to Proposition~\ref{prop.resolvent} equivalent to resolvent estimates for $\dot{A}$). 
    \subsection{Reduction to high frequency observation estimates}
  In this section, we show that for $m\geq 0$, stabilisation is equivalent to semi-classical observation estimates (see~\cite{Mi12}).
  \begin{prop}\label{prop.B.2} Assume that $0\leq a\in L^\infty$ is non trivial ($\int_{M} a >0$). Then stabilisation holds for~\eqref{eq.amortie} if and only if 
  \begin{equation}\label{obs}
  \begin{gathered}
 \exists h_0>0; \forall 0<h<h_0, \forall  (u,f) \in H^2(M)\times L^2(M), (h^2 \Delta + 1) u =f,  \\
 \| u\|_{L^2(M)} \leq C \bigl( \| a^{1/2} u\|_{L^2} + \frac {1}{h} \| f\|_{L^2}\Bigr).
 \end{gathered}
 \end{equation}
 \end{prop}
  We prove the proposition for $m=0$. The proof for $m\not \equiv 0$ is similar (slightly simpler as we do not have to work with the operator $\dot{A}$ but can stick with $A$).
 From Proposition~\ref{prop.resolvent}, stabilisation is equivalent to the fact that the resolvent $(\dot{A} - \lambda)^{-1}$ is bounded on $\mathbb{C}^\delta$. Since $\dot{A}$ is maximal dissipative, its resolvent is defined (and bounded) on any domain $\mathbb{C}^{-\epsilon}$ ($\epsilon >0$). We deduce that  it is equivalent to prove that it is uniformly bounded on $i\mathbb{R} $ (and consequently by perturbation on a $\delta$ neighborhood of $i\mathbb{R} $). Since 
 $$ (\dot{A}- \lambda) =  ( 1+ (1-\lambda)(\dot{A}- 1)^{-1})( \dot{A}- 1),$$
 and $(\dot{A}- 1)^{-1}$ is compact (see Corollary~\ref{coroA.4}) on $\dot{\mathcal{H}}$ (see Corollary~\ref{coroA.4}), the operator $( 1+ (1-\lambda)(\dot{A}- 1)^{-1} )$ is Fredholm with index $0$ and consequently, $\dot{A}- \lambda$ is invertible iff it is injective. As a consequence, stabilisation is equivalent to the following {\em a priori} estimates
\begin{equation}\label{resolv-wave}
 \exists C>0; \forall \lambda \in \mathbb{R}, U \in \dot{H}^2 \times H^1, F \in \dot{H}^1 \times L^2,  (\dot{A} - i \lambda) U = F \Rightarrow \| U\|_{\dot{\mathcal{H}} } \leq C\| F\|_{\dot{\mathcal{H}} }.
 \end{equation}
   \subsubsection{High frequency resolvent estimates imply stabilisation}
   We argue by contradiction. We assume~\eqref{obs} holds and assume that~\eqref{resolv-wave} does not hold. Then there exists sequences $(\lambda_n), (U_n), (F_n)$ such that 
  $$ (\dot{A} - i \lambda_n) U_n = F_n, \qquad  \|U_n\|_{\dot{\mathcal{H}} } > n \| F_n\|_{\dot{\mathcal{H}} }.$$ 
  Since $U_n \neq 0$, we can assume $\| U_n\|_{\dot{\mathcal{H}} }=1$. 
  Extracting subsequences we can also assume that $\lambda_n \to \lambda \in \mathbb{R} \cup \{ \pm \infty\}$ as $n\to \infty $. We write 
  $$U_n= \begin{pmatrix} \dot{u}_n \\ v_n\end{pmatrix}, F_n= \begin{pmatrix} \dot{f}_n \\ g_n\end{pmatrix},$$ and distinguish according to three cases 
 
 \noindent{--} {Zero frequency: $\lambda =0$. }
  In this case, we have 
 $$ \dot{A} U_n = o(1)_{\dot{\mathcal{H}}} \Leftrightarrow \Pi v_n = o(1)_{\dot{H}^1}, \quad \Delta \dot{u}_n - a v_n = o(1)_{L^2}.$$
 We deduce that there exists $c_n \in \mathbb{C}$ such that 
 $$ v_n - c_n = o(1)_{H^1}, \qquad \Delta u_n - a c_n = o(1)_{L^2}.$$
 But 
 $$ \int_M \Delta u_n =0 \Rightarrow c_n \int_M a = o(1) \Rightarrow c_n = o(1).$$
 As a consequence, we get $ v_n = o(1) _{L^2}$ and $\Delta u_n = o(1)_{L^2} \Rightarrow \dot{u} _n = o(1) _{\dot{H} ^1}$. This contradicts  $\| U_n\|_{\dot{\dot{\mathcal{H}} }}=1$.
 
\noindent{--}  {Low frequency: $\lambda \in \mathbb{R}^*$.}
  In this case, we have 
 $$ (\dot{A}-i \lambda)  U_n = o(1)_{\dot{\mathcal{H}}} \Leftrightarrow \Pi v_n  - i \lambda \dot{u} _n= o(1)_{\dot{H}^1}, \quad \Delta \dot{u}_n - ( i \lambda +a)  v_n = o(1)_{L^2}.$$
 We deduce 
 $$ \Delta v_n - i\lambda ( a+ i\lambda) v_n = o(1)_{L^2} + \Delta (o(1) _{\dot{H}^1} ) = o(1)_{H^{-1}}.$$
 Since $(v_n)$ is bounded in $L^2$, from this equation, we deduce that $\Delta v_n$ is bounded in $H^{-1} $ and consequently $v_n$ is bounded in $H^1$. Extracting another subsequence, we can assume that $v_n$ converges in $L^2$ to $v$ which satisfies
 $$ \Delta v + \lambda ^2 v - i\lambda av =0.$$
 Taking the imaginary part of the scalar product with $v$ in $L^2$ gives (since $\lambda \neq 0$) 
 $ \int_{M} a | v|^2 =0$, and consequently $av =0$ which implies that $v$ is an eigenfunction of the Laplace operator. But since the zero set of non trivial eigenfunctions has Lebesgue measure $0$ in $M$,  $av=0$ implies that $v=0$ (and consequently $v_n = o(1)_{L^1}$). Now, we have 
 $$ \Delta \dot{u}_n = (i\lambda +a ) v_n + o(1)_{L^2} = o(1)_{L^2} \Rightarrow \dot{u}_n = o(1) _{\dot{H}^1}.$$
 This contradicts $\| U_n\|_{\dot{\mathcal{H}} }=1$.
 
  \noindent{--}  {High frequency $\lambda_n \rightarrow \pm \infty$. }
 We study the case $\lambda_n \rightarrow + \infty$, the other case is obtained by considering $\overline{U_n}$. Let $h_n = \lambda_n^{-1}$. 
 \begin{equation}
  \begin{gathered}
 (\dot{A}- i\lambda_n) U_n = o(1)_{\dot{\mathcal{H}}} \Leftrightarrow -i \lambda_n\dot{u}_n + \Pi v_n = o(1)_{\dot{H}^1}, \quad \Delta \dot{u}_n - (i \lambda_n+a) v_n = o(1)_{L^2}\\
 \Leftrightarrow \dot{u}_n = -i h_n \Pi v_n + o(h_n)_{\dot{H}^1}, \qquad( h_n ^2 \Delta  + 1 - i h_n a)v_n = o(h_n )_{L^2}+ o( h_n ^2)_{H^{-1}}\label{absurde}
 \end{gathered}
 \end{equation}
 
 To conclude in this regime, we need 
 \begin{lemm}
 The observation inequality~\eqref{obs} implies the more general
  \begin{gather}\notag
 \exists h_0>0; \forall 0<h<h_0, \forall  (u,f_1, f_2) \in H^2(M)\times L^2(M)\times H^{-1} (M), (h^2 \Delta + 1) u =f_1+ f_2,  \\
 \| h \nabla_x u \|_{L^2} + \| u\|_{L^2(M)} \leq C \bigl( \| a^{1/2} u\|_{L^2} + \frac {1}{h} \| f_1\|_{L^2}+\frac {1}{h^2} \| f_2\|_{H^{-1}}\Bigr). \label{obsbis}
 \end{gather}
 \end{lemm}
 \begin{proof}
 Let $P^\pm_h= h^2 \Delta + 1 \pm i ha $ defined on $L^2$ with domain $H^2$. Writing 
 $$ P_h^\pm = ( 1 + (2 \pm iha) (h^2 \Delta - 1)^{-1} ) ( h^2 \Delta -1),$$ and since $( h^2 \Delta -1)^{-1}$ is compact on $L^2$, we deduce that $( 1 + (2 \pm iha) (h^2 \Delta - 1)^{-1} )$ is Fredholm with index $0$, hence $P^\pm_h$ is invertible iff it is injective. On the other hand we have 
 $$ h\| a^{1/2} u\|_{L^2}^2  =\pm \Im (P^\pm_h u , u)_{L^2} \leq \|P_h^\pm u\|_{L^2} \| u\|_{L^2},$$
 which combined with~\eqref{obs} implies ($ (h^2 \Delta +1) u = P^\pm_h u\mp iha u$)
 \begin{multline}
  \| u\| ^2_{L^2} \leq C \bigl( \| a^{1/2} u\|^2_{L^2} + \frac {1}{h^2} (\| P^\pm_h u\|^2_{L^2} + h^2 \| au \|^2_{L^2})\Bigr) \\
  \leq \frac{C'}{h} \|P^\pm _h u\|_{L^2} \| u\|_{L^2} + \frac {C'}{h^2} \| P^\pm_h u\|^2_{L^2} \Rightarrow   \| u\| _{L^2}\leq  \frac {C''}{h} \| P^\pm_h u\|_{L^2}  
  \end{multline}
  Since 
  \begin{equation}\label{ipp} 
   |\| u\|_{L^2}^2 - \| h \nabla_x u \|_{L^2}^2| = |\Re (P^\pm_h u , u)_{L^2}|  \leq \|P^\pm _h u\|_{L^2} \| u\|_{L^2},
   \end{equation}
 We deduce that $P^\pm_h$ is injective hence bijective from $H^2$ to $L^2$ with inverse bounded by $C''/h$ from $L^2$ to $L^2$ and by $C/h^2$ from $L^2$ to $H^1$. We now proceed by duality to obtain~\eqref{obsbis}. The adjoint of $P_h^\pm$ is $P_h^\mp$ and is consequently bounded from $H^{-1}$ to $L^2$ by $C/h^2$. 
 Using again the identity~\eqref{ipp} we get that
 $$P_h^\pm  u = f_1+ f_2\Rightarrow \| h \nabla_x u \|_{L^2} + \| u\|_{L^2(M)} \leq \frac C h \|f_1\|_{L^2} + \frac C {h^2} \| f_2\|_{H^{-1}}.$$
 Finally 
 $$(h^2 \Delta + 1 ) u = f_1+ f_2\Rightarrow P_h^+ u = ia h u + f_1+ f_2,$$
 and we get 
 \begin{multline}
  \| h \nabla_x u \|_{L^2} + \| u\|_{L^2(M)} \leq C \bigl(  \frac {1}{h} \| iha u+ f_1\|_{L^2}+\frac {1}{h^2} \| f_2\|_{H^{-1}}\Bigr)\\
  \leq C' \bigl( \| a^{1/2} u\|_{L^2} + \frac {1}{h} \|  f_1\|_{L^2}+\frac {1}{h^2} \| f_2\|_{H^{-1}}\Bigr)
  \end{multline}
     \end{proof}
   
We now come back to our sequence satisfying~\eqref{absurde}. From~\eqref{obsbis}, \eqref{absurde}  implies 
 $$ \| h_n \nabla_x v_n \|_{L^2} + \| v_n \|_{L^2} = o(1)_{n\rightarrow + \infty},$$
 and in turn 
 $$ \| \nabla_x u_n \|_{L^2} =o(1)_{n\rightarrow + \infty}.$$
 This contradicts  
 $\| U_n\|_{\dot{\mathcal{H}} }=1$.
\subsubsection{Stabilisation imply resolvent estimates}
Consider now $U= \begin{pmatrix} \dot u\\ v\end{pmatrix}, F= \begin{pmatrix} \dot f\\ g\end{pmatrix}$ such that 
$$ (\dot{A} - i \lambda)  U = F \Leftrightarrow -i \lambda \dot{u} + \Pi v = \dot{f} \text{ and } (\Delta v +\lambda^2-i\lambda a) v= i \lambda g + \Delta f.$$
From~\eqref{obs} with  $h = \lambda^{-1}$, we get
$$\| v\|_{L^2} + \|h \nabla_x v\|_{L^2} \leq C \| g\|_{L^2} + C \| \Delta f\|_{H^{-1}} \leq C ( \| g\|_{L^2} + C \| \nabla_x f\|_{L^2}),$$
and also 
$$ \| \nabla_x u\|_{L^2} = h \| \nabla_x (v-f)\| \leq    C ( \| g\|_{L^2} + C \| \nabla_x f\|_{L^2}).$$
\section{Caracterization of stabilisation}\label{app:C}
Here we shall prove that the properties (1), (2), (3) and (4) of the Introduction are equivalent. $(2) \Rightarrow (1)$ is trivial. To show $(1) \Rightarrow (3)$ we fix $T$ such that $f(T) \leq 1/2$. Then since 
$$ E_m(u)(T) = E_m(u) (0) - \int_0^T \int_M a(x) |\partial_t u \|^2 (t,x) dx dt\leq \frac 1 2 E_m(u) (0),$$
we deduce 
$$ E_m(u) (0) \leq 2\int_0^T \int_M a(x) |\partial_t u \|^2 (t,x) dx dt,$$
which is (3). Conversely, if (3) is satisfied, we get 
$$  E_m(u)(T) = E_m(u) (0) - \int_0^T \int_M a(x) |\partial_t u \|^2 (t,x) dx dt\leq(1- \frac 1 C) E_m(u) (0).$$
Let $\delta = (1- \frac 1 C) <1$. Applying the previous estimate between $0$ and $T$, then $T$ and $2T$, etc, we get 
$$E_m(u)(nT)\leq \delta ^n E_m(u)(0),$$ hence the exponential decay along the discrete sequence of times $nT$. Finally, writing $nT \leq t < (n+1) T$, we get 
$$ E_m(u)(t) \leq E_m(u)(nT)\leq \delta^n E_m(u)(0) \leq e^{ \log(\delta) (\frac t T -1)} E_m(u)(0),$$
which is (2). It remains to prove that (3) and (4) are equivalent. We shall actually prove that if (3) holds for some $T>0$, then (4) holds for the same $T>0$.  Let us fix $T>0$ and assume that (4) does not hold, i.e. there exists sequences $(u_0^n, u_1^n)\in H^1\times L^2$ such that the corresponding solutions to the undamped wave equation~\eqref{eq.nonamortie} satisfy 
$$ E_m (u_n) (0) > n \int_0 ^T \int_M a(x) |\partial_t u^n|^2 (t,x) dtdx .$$ 
 This implies that $u_n$ is non identically $0$ and dividing $u_n$ by $\sqrt{ E_m (u_n)(0)})$, we can assume that $E_m (u_n)(0)=1$, and 
 \begin{equation}\label{petit}
 \int_0 ^T \int_M a(x) |\partial_t u_n|^2 (t,x) dtdx\leq \frac 1 n.
 \end{equation}
 Consider now $(v_n)$ the sequence of solutions to the damped wave equation~\eqref{eq.amortie}, with the same initial data $(u_0^n, u_1^n)$, and $w_n = u_n - v_n$ solution to 
 $$  (\partial_t ^2 - \Delta+a \partial_t + m ) w_n =- a \partial_t u_n, \quad (w_n\mid_{t=0}, \partial_t w_n \mid_{t=0} ) = (0,0) \in (H^1\times L^2) (M).$$
 From Duhamel formula and~\eqref{petit} we deduce 
 \begin{multline}
  \| (w_n, \partial_t w_n)\|_{L^\infty((0,T); H^1(M)\times L^2(M))} \leq \| a \partial_t u_n\| _{L^1(0,T); L^2(M)} \\
  \leq \|a\|_{L^\infty}^{1/2} \| a^{1/2} \partial_t u_n\| _{L^1(0,T); L^2(M)}= o(1)_{n\rightarrow + \infty}.
  \end{multline}
 We deduce 
 $$ E_m (v_n) (0) = 1+ o(1)_{n\rightarrow + \infty}$$ 
 and 
 $$ \int_0 ^T \int_M a(x) |\partial_t v_n|^2 (t,x) dtdx= \| a^{1/2} v_n \| _{L^2((0,T) \times M)}=o(1)_{n\rightarrow + \infty},$$
  which implies that (3) does not hold. As a consequence, we just proved $(3) \Rightarrow (4)$. The proof of $(4) \Rightarrow (3)$ is similar.

 \def\cprime{$'$} \def\cprime{$'$}

\end{document}